\documentclass[a4paper,11pt,proof]{article}

\usepackage{hyperref,proof}

\usepackage{latexsym}
\usepackage{amsmath}
\usepackage{amscd}
\usepackage{amssymb,enumerate}

\usepackage{url}

\usepackage{amsthm}

\def\RCAo{\mathsf{RCA_0}}

\def\RCA{\mathsf{RCA_0}}

\def\ACA{\mathsf{ACA_0}}

\def\N{\mathbb{N}}

\newcommand{\mc}[1]{\mathcal{#1}}

\def\P2{\Pi^1_2}

\newcounter{menum}
{\begin{enumerate}%
\setcounter{enumi}{#1}}%
{\setcounter{menum}{\value{enumi}}\end{enumerate}}

\newtheorem{thm}{Theorem}[section]
\newtheorem{theorem}[thm]{Theorem}

\newtheorem{claim}{Claim}[thm]
\newtheorem*{claim*}{Claim}

\newtheorem{proposition}[thm]{Proposition}
\newtheorem{lemma}[thm]{Lemma}
\newtheorem{corollary}[thm]{Corollary}

\theoremstyle{definition}
\newtheorem{defi}{Definition}[section]
\newtheorem{definition}[defi]{Definition}

\newtheorem{remark}[thm]{Remark}

\newtheorem{question}[defi]{Question}
\newtheorem{example}[defi]{Example}

\usepackage{tikz}
\usetikzlibrary{matrix}
\usepackage{centernot}

\usepackage[mode=multiuser,status=draft,lang=english]{fixme}
\fxsetup{theme=color}

\FXRegisterAuthor{K}{aK}{\color{yellow}Keita}
\FXRegisterAuthor{S}{aS}{\color{violet}Silvia}
\FXRegisterAuthor{E}{aE}{Emanuele}

\newcommand\NN{\mathbb{N}}

\newcommand{\is}[1]{\mathsf{I}\Sigma^0_{#1}}

\newcommand{\cl} {\mathrm{cl}}
\newcommand\mcal{\mathcal}
\newcommand\arity{\mathrm{arity}}
\newcommand\interp[1]{[\![#1]\!]}

\newcommand\WO{\mathrm{WO}}
\newcommand\Tot{\mathrm{Tot}}

\newcommand{\ap}[1]{\langle #1 \rangle}
\newcommand{\bp}[1]{\left\lbrace #1 \right\rbrace}

\newcommand{\arcd}[2] {#1\xrightarrow{\downarrow} #2}
\newcommand{\arce}[2] {#1\xrightarrow{\Downarrow} #2}
\newcommand{\var}{\mathrm{Par}}
\newcommand\fun{\mathrm{Fun}}
\newcommand\op{\mathrm{Op}}
\newcommand\aexp{\mathrm{AExp}}
\newcommand\bexp{\mathrm{BExp}}
\newcommand\eexp{\mathrm{Exp}}
\newcommand\ddef{\mathrm{Def}}
\newcommand\prog{\mathrm{Prog}}
\newcommand\term{\mathrm{Term}}

\newcommand\temp{\mathrm{temp}}
\newcommand\esp{\mathrm{exp}}

\newcommand{\mboxb}[1]{\mbox{ \textbf{#1} }}
\newcommand{\mbb}[1]{\mbox{ \textbf{#1} }}
\newcommand{\mb}{\mathbf}
\newcommand\tb{\textbf}
\newcommand{\conc}{{{}^\smallfrown}}
\newcommand{\last}{\text{last}}
\newcommand{\pos}{\text{pos}}

\newcommand{\ev}{\downarrow}
\newcommand{\nev}{\uparrow}
\newcommand{\biimp}{\Longleftrightarrow}

\usepackage{authblk}

\title{The strength of SCT soundness}

\author[1]{Emanuele Frittaion}
\author[2]{Florian Pelupessy}
\author[3]{Silvia Steila}
\author[4]{Keita Yokoyama}

\affil[1]{ \footnotesize  Department of Mathematics,
	FCUL - Universidade de Lisboa, Portugal}
\affil[1]{	\textit{emanuelefrittaion@gmail.com}}
\affil[2]{\footnotesize  Mathematical Institute, 
	Tohoku University,  Japan}
\affil[2]{	\textit{florian.pelupessy@operamail.com}}
\affil[3]{\footnotesize  Institute of Computer Science, University of Bern, Switzerland}
\affil[3]{	\textit{steila@inf.unibe.ch}}
\affil[4]{\footnotesize  School of Information Science, Japan Advanced Institute of Science and Technology, Japan}
\affil[4]{	\textit{y-keita@jaist.ac.jp}}
\date{}

\begin{document}
\maketitle
\tableofcontents

\newpage

\begin{abstract}
	In this paper we continue the study, from Frittaion, Steila and Yokoyama (2017), on size-change termination in the context of Reverse Mathematics. We analyze the soundness of the SCT method. In particular, we prove that the statement ``any program which satisfies the combinatorial condition provided by the SCT criterion is terminating'' is equivalent to $\WO(\omega_3)$ over $\RCA$.
\end{abstract}


\section{Introduction}

Informally, a recursive definition of a function has the SCT property if, in every infinite sequence of calls, there is some infinite sequence of parameter values which is weakly decreasing and strictly decreasing infinitely many times. If the parameter values are well-ordered, as in the case of natural numbers with the natural  ordering, there cannot be such a sequence. Thus, the SCT property is a sufficient condition for termination. 

The SCT property concerns the semantics of the program. In \cite{Jones1} Lee, Jones and Ben-Amram provided an alternative property, equivalent to being SCT, but which can be statically verified from the definition of the program. Indeed, they proved the following:

\begin{theorem}[SCT criterion]\label{sct criterion}
	Let $\mathcal{G}$ be a description of a program $P$. Then $\mathcal{G}$ is SCT  iff every idempotent $G\in{\sf cl}(\mathcal{G})$ has an arc $\arcd{x}{x}$. 
\end{theorem}

Here, ${\sf cl}(\mathcal{G})$ is a set of graphs, which can be extracted directly from the code of the program $P$. We refer to Section \ref{subsection:sct} for definitions.

The SCT criterion leads us to consider two distinct, although classically equivalent, properties.  For convenience of exposition, we use the following terminology.  See Section \ref{definitions} for definitions. 

\begin{definition}
	\begin{itemize}
		\item $\mathcal{G}$ is \tb{MSCT} (Multipath-Size-Change terminating) if $\mathcal{G}$ is SCT. 
		\item $\mathcal{G}$ is \tb{ISCT} (Idempotent-Size-Change terminating) if every idempotent $G\in{\sf cl}(\mathcal{G})$ has an arc $\arcd{x}{x}$. 
	\end{itemize} 
\end{definition}

With this terminology at hand, we outline the following three-step argument from \cite{Jones1} to prove the termination of a first order functional program $P$:
\begin{itemize}
	\item Verify that $P$ is ISCT;
	\item Apply the SCT criterion to prove that $P$ is MSCT;
	\item Derive the termination of $P$ from the fact that ``every MSCT program terminates''. 
\end{itemize}

Since the Ackermann function is ISCT provably in $\RCA$ (see Section \ref{Section:Ackermann}), a natural question arises: in which theory can we carry out the above argument? Specifically, in which theory can we prove the SCT criterion? Similarly, in which theory can we prove that every MSCT program terminates?

It is clear that this cannot be done in weak theories, such as $\RCA$, which do not prove the termination of the Ackermann function. In \cite{FSY2017} we studied the strength of the SCT criterion. We proved that the SCT criterion follows from a special instance of Ramsey's theorem for pairs, which turns out to be  equivalent to $\is2$ over $\RCA$. In the present paper we focus our attention on the second question, i.e., the soundness of the MSCT principle. Moreover, we investigate in which theory we can prove directly the termination of programs which are ISCT, without applying the SCT criterion.  
We thus consider the following two soundness statements.

\begin{theorem}[ISCT (resp.\ MSCT) Soundness]\label{soundness}
	Let $\mathcal{G}$ be a safe description of a program $P$. If $\mathcal{G}$ is ISCT (resp. MSCT) then  $P$ is terminating.
\end{theorem}

Following  standard notation (e.g., \cite{SOSOA}), $\WO(\alpha)$ states that the linear ordering $\alpha$ is well-ordered. In this paper we show that over $\RCA$,
\begin{itemize}
 \item  ISCT soundness $=$ $\WO(\omega_3)$  $\geq $ MSCT soundness $ > \WO(\omega_2)$,
\end{itemize}
where $\omega_2=\omega^\omega$ and $\omega_3=\omega^{\omega^\omega}$.

One direction of the SCT criterion is provable in $\RCA$. That is, within $\RCA$, every MSCT description is also ISCT. Therefore, provably in $\RCA$,  ISCT soundness implies MSCT soundness. Moreover, over $\RCA+\is2$, where ISCT and MSCT are equivalent notions, we have 

\begin{itemize}
	\item MSCT soundness $=$ $\WO(\omega_3)$.
\end{itemize}

It is still an open question whether the inequality is strict over $\RCA$: 

\begin{question}\label{question}
Is MSCT soundness equivalent to $\WO(\omega_3)$  over $\RCA$?
\end{question}

As a consequence of our analysis, we provide ordinal bounds for the termination of first order functional programs which are ISCT and, in particular, a new proof of primitive recursive bounds for the tail-recursive ISCT programs. 

\subsection{Reverse mathematics}

Reverse mathematics is a research program in mathematical logic and foundations of mathematics. We refer to Simpson \cite{SOSOA} and Hirschfeldt \cite{Hirschfeldt} for a general overview. The goal is to assess the relative logical  strength of theorems from ordinary (non set-theoretic) mathematics, thus making sense of statements like Theorem A is stronger than Theorem B or Theorem A and Theorem B are equivalent. This program is carried out in second-order arithmetic, a formal system for natural numbers and sets of natural numbers, which is expressive enough to accommodate large parts of ordinary mathematics. Given a theorem $A$, one looks for the minimal subsystem $\Xi$ needed to prove $A$, that is, $A$ follows from $\Xi$ and all the axioms from $\Xi$ are provable from $A$ over a base system $\Xi_0$. 

 
The most important subsystems of second-order arithmetic are obtained by restricting comprehension and induction to some class $\Gamma$ of formulas. 
\begin{itemize}
	\item $\Gamma$-Comprehension: $\exists X\forall n(n\in X\leftrightarrow \varphi(n))$, for $\varphi\in\Gamma$	
	\item $\Gamma$-Induction: $\varphi(0)\land \forall n(\varphi(n)\rightarrow\varphi(n+1))\rightarrow\forall n\varphi(n)$, for $\varphi\in \Gamma$
\end{itemize} 
In practice, one uses comprehension to define a set and induction to verify that the set thus defined has the required properties. 

Let us briefly recall the definition of $\Sigma^0_n$ formulas. Bounded quantifiers are of the form $\forall x<t$ and $\exists x<t$, with $x$ being a number variable and $t$ a number term. A formula $\varphi$ is $\Sigma^0_n$ if it is of the form $\exists x_1\forall x_2\ldots Qx_n\theta$, where the $x_i$'s are number variables and all quantifiers in $\theta$ are bounded. Similarly, $\varphi$ is $\Pi^0_n$ if it is of the form $\forall x_1\exists x_2\ldots Qx_n\theta$, with $x_i$'s and $\theta$ as above. A formula is arithmetical if it is $\Sigma^0_n$ for some $n$. Note that $\theta$ can contain set variables. 
 
In this paper we are mainly concerned with the base system $\RCA$ (Recursive Comprehension Axiom), the system $\ACA$ (Arithmetical Comprehension Axiom), the induction schemes $\is{n}$ (induction for $\Sigma^0_n$ formulas), and the principles $\WO(\omega_n)$ ($\omega_n$ is well-ordered), where $n$ is a standard natural number.

$\RCA$ consists of the usual first-order axioms of Peano arithmetic, without induction, plus comprehension and induction, restricted to $\Delta^0_1$ and $\Sigma^0_1$ formulas respectively. Roughly,  $\RCA$ proves that all computable sets exist. 
$\ACA$ is obtained from $\RCA$ by adding arithmetical comprehension (comprehension for all arithmetical formulas). Roughly, $\ACA$ proves that the Turing jump of every set exists. 

The statements $\is{n}$ and $\WO(\omega_n)$ form two intertwined hierarchies  below $\ACA$:
\begin{center}
\begin{tikzpicture}
  \matrix (m) [matrix of math nodes]
  {
\ACA	&\longrightarrow 	&\is{n+3} 			& \longrightarrow &\is{n+2} 				& 							&	\\
		& 					&\downarrow		&  				 &\downarrow			& 						      	&	\\
		&					&\WO(\omega_{n+3}) & \longrightarrow & \WO(\omega_{n+2})	& \centernot\longleftarrow 	& \RCA	\\
  };     
\end{tikzpicture}
\end{center}


The implications in the picture are strict. It is known that $\WO(\omega^\omega)$ is equivalent, over $\RCA$, to the totality of the (relativized) P\'{e}ter-Ackermann function $A^2_f$ \cite{KK}. In this paper we show that $\WO(\omega_3)$ is equivalent to the totality of all $A^n_f$, where $A^n_f$ is a natural generalization of $A^2_f$ to $n$ arguments. 

As mentioned earlier, the original proof of the SCT criterion makes use of Ramsey's theorem for pairs \cite{Ramsey}. This theorem states that for every coloring, on the edges of the complete graph on countably many nodes, in $k$ colors, there exists an infinite homogeneous set. I.e., there exists an infinite subset of the nodes such that any two elements in this subset are connected in the same color. Ramsey's theorem for pairs, in symbols RT$^2$, is a key principle in reverse mathematics (e.g., see \cite{SOSOA}, \cite{Hirschfeldt}, \cite{CJS}, \cite{CSY2012}). 

\section{SCT framework}\label{definitions}

\subsection{Syntax}

\begin{align*}
	x\in \var  & \  \text{ parameter}\\
	f \in \fun   & \ \text{ function identifier }\\
	f  \in \op   & \ \text{  primitive operator}\\
	a\in \aexp  & \ \text{ arithmetic expression}  \\
	&::= \ x \mid x+1 \mid x-1 \mid f(a,\ldots,a) \\
	b \in \bexp   & \ \text{ boolean expression}       \\
	&::= \  x=0 \mid x=1 \mid x<y \mid x\leq y \mid b\land b \mid b\lor b \mid \neg b \\
	e \in \eexp & \ \text{ expression} \\ 
	&::= \  a \mid \mboxb{if } b \mboxb{ then } e \mboxb{ else } e \\
	d \in \ddef &  \ \text{ function definition} \\
	&::=\ f(x_0,\ldots,x_{n-1})= e \\
	P\in \prog & \ \text{ program} \\
	&::= \ d_0,\ldots,d_{m-1} 
\end{align*}
A program $P$ is a list of finitely many  defining equations  $f(x_0,\ldots,x_{n-1})=e^f$, where $f\in\fun$ and $e^f$ is an expression, called the \tb{body} of $f$. Let  $x_0,\ldots,x_{n-1}$ be the \tb{parameters} of $f$, denoted $\var(f)$, and let $n$ be the \tb{arity} of $f$, denoted  $\arity(f)$. Function identifiers on the left-hand side of each equation are assumed to be distinct from one another.   By $\fun(P)$ we denote the set of function identifiers occurring in $P$ and by $\op(P)$ the set of primitive operators occurring in $P$. We usually suppress the reference to $P$ whenever it is clear from the context. 
The \tb{entry} function $f_{0}$ is the first in the program list. The idea is that  $P$ computes the partial function $f_{0}:\N^{\arity(f_{0})}\to\N$. 

\begin{example}
The following program computes the well-known P\'{e}ter-Ackermann function: 
\[
\begin{split}
A(x,y) = &\mbb{ if } x=0 \mbb{ then } y+1\\
&\mbb{ else if } y=0 \mbb{ then }  A(x-1,1)\\
&\qquad \mbb{ else }  A(x-1,A(x,y-1))\\
\end{split}
\]
\end{example}

\subsection{Semantics} The standard semantics for first order functional programs  is \emph{denotational semantics} (see, e.g.,  Lee, Jones and Ben-Amram \cite{Jones1}).
Another possible choice is \emph{operational semantics} (as defined in, e.g., \cite{Winskel93}). In our framework we find it natural and convenient to interpret programs as \emph{term rewriting systems}.

\tb{Notation}: We use $\mb u, \mb v$ for tuples of natural numbers and $\mb s, \mb t$ for tuples of terms.

In general a term rewriting system is a set of rules, i.e., objects of the form $s\to t$ for $s,t$ terms. We build up our terms by using natural numbers and function symbols (function identifiers and primitive operators). In particular, terms do not contain if-then. 
\[  t\in\term ::= u\in\N \mid f(t,\ldots,t) \]

Boolean expressions are decidable and we can think of a boolean expression $b$ with parameters in $x_0,\ldots,x_{n-1}$ as a primitive operator whose intended interpretation is a boolean function $\interp b:\NN^n\to 2$. For instance, we read $b(x_0,\ldots,x_{n-1})$ as $b(x_0,\ldots,x_{n-1})=0$. Symbols such as $0,1,+,-$  have the intended interpretation. We do not distinguish notationally between these symbols and their interpretation, relying on context to distinguish the two. For instance, $x+1$ is an expression if $x$ is a parameter, or the successor of $x$ if $x$ is a natural number.
 
Given an expression $e$ with parameters in $\mb x=x_0,\ldots x_{n-1}$ and a tuple  $\mathbf{u}\in\N^n$, we want to evaluate the expression $e$ on $\mb u$ and return a term $e(\mb u)$ (\footnote{Ultimately, this comes down to substitute $\mb x$ with $\mb u$ in the unique arithmetic subexpression of $e$ which is determined by the boolean tests for $\mb u$. That is, if $e_0,\ldots, e_{k-1}$ are the maximal arithmetic subexpressions of $e$, in the sense that they are not proper subexpressions of any arithmetic expression of $e$, then we have $e(\mb u)=e_i[\mb x/\mb u]$, where $i$ is uniquely determined.}). We can easily define $e(\mb u)$ by recursion on the construction of $e$ as follows:
\begin{itemize}
	\item  $x_i(\mb u)=u_i$, $(x_i+1)(\mb u)=u_i+1$, and $(x_i-1)(\mb u)=u_i-1$ if $u_i>0$, $0$ otherwise. 
	\item If  $e=f(e_0,\ldots, e_{k-1})$, then $e(\mb u)=f(e_0(\mb u),\ldots,e_{k-1}(\mb u))$. 
	\item  If $e=\mboxb{if } b \mboxb{ then } e_0 \mboxb{ else } e_1$, 
	then \[ e(\mb u)= \begin{cases} e_0(\mb u) & \text{if $\interp b(\mb u)=0$} \\ e_1(\mb u) & \text{otherwise.}\end{cases}\] 
\end{itemize}

Every subterm of a given term $t$ has a \tb{position} $\sigma$. We can use sequences of natural numbers to determine the position of a subterm. For instance, if $t=f(g(2),4)$,  then $t$ has position $\langle\rangle$ and $g(2)$ has position  $\langle 0\rangle$. Formally:

\begin{definition}[position]
	Let $s$ be a subterm of $t$ and  $\sigma$ be a sequence of natural numbers.  We say that the position of $s$ in $t$ is $\sigma$ (in symbols $\pos(s,t)=\sigma$) if $t=f(t_0,\ldots,t_{k-1})$ and one of the following holds: 
	\begin{itemize}
		\item $s=t$ and $\sigma = \ap{}$
		\item $s$ is a subterm of $t_i$ and $ \sigma = \ap{i} \conc \pos(s, t_i) $.
	\end{itemize}
	We write $t|_\sigma$ for the unique subterm of $t$ in position $\sigma$ (if it exists).
\end{definition}

Similarly,  every subexpression of a given expression $e$  has a \tb{position} $\tau\in\NN^{<\NN}$, and we write $e|_\tau$ for the unique subexpression of $e$ in position $\tau$ (if it exists). Formally:
\begin{definition}[position] Let $e,e'$ be expressions and $\tau\in\N^{<\N}$. We say that $e'$ has position $\tau$ in $e$ if one of the following holds:
	\begin{itemize}	
		\item 	$e=e'$ and $\tau=\langle\rangle$;
		\item   $e=f(e_0,\ldots,e_{n-1})$, $\tau=\ap{i}\conc \tau'$, and $e'$ has position $\tau'$ in $e_i$;
		\item $e= \mboxb{if } b \mboxb{ then } e_0 \mboxb{ else } e_1$, $\tau= \ap{i}\conc \tau'$ with $i<2$, and $e'$ has position $\tau'$ in $e_i$.
	\end{itemize}
\end{definition}

Fix an an interpretation $\interp \cdot$ of all primitive operators, that is, for all $f\in\op$ let \[\interp f\colon\N^{\arity(f)}\to\N.\] 

We are now ready to define, given a program $P$, a term rewriting system $T_P$.

\begin{definition}[rules and reduction]
A \tb{rule} is of the form $f(\mb u)\to e^f(\mb u)$ for $f\in\fun$ or $f(\mb u)\to {\interp f}(\mb u)$ for $f\in\op$. 	A  \tb{one-step reduction} $t \rightarrow_P s$ is given by replacing the leftmost subterm $f(\mb u)$ of $t$ according to the rule. We write $s= t[f(\mb u)]_\rho$, where $\rho$ is the position of $f(\mb u)$ in $t$.
\end{definition}

Note that $\rightarrow_P$ is decidable.

\begin{example}
Suppose we want to compute $A(2,3)$, the value of the  P\'{e}ter-Ackermann function at $(2,3)$. According to the definition we have:
\[
A(2,3)\to A(1,A(2,2))\to A(1,A(1,A(2,1)))\to \ldots
\]
\end{example}

Use $\rightarrow_P^*$ to denote the reflexive transitive closure of $\rightarrow_P$.

\begin{definition}[state transition]
	For $f,g\in\fun$ and $\tau\in\NN^{<\N}$, define a \tb{state transition} $(f,\mathbf{u})\xrightarrow{\tau}(g, \mb v)$ by $e^f|_\tau(\mb u) = g(\mb s)$ and $s_i \rightarrow_P^*  v_i$ for all $i<\arity(g)$. 
\end{definition} 

For every subterm $s$ of $e^f(\mb u)$ there exists a unique position $\tau$ in $e^f$ such that $s=e^f|_\tau(\mb u)$. We say that $\langle \tau, f, g\rangle$ is a \tb{call} from $f$ to $g$ and write $\tau:f\to g$. 
It is worth  noticing that there are only finitely many $\tau$'s and hence finitely many calls $\tau:f\to g$. This apparently obvious fact is essential for the SCT criterion (from ISCT to MSCT) and for the SCT soundness.

We can extend the state transition relation to $(f,\mathbf{u})\xrightarrow{\tau}(g, \mb t)$ by the same definition.

The relations $\to_P^*$ and $\xrightarrow{\tau}$ are $\Sigma^0_1$. In particular, the latter is $\Sigma^0_1$ by $\sf{B}\Sigma^0_1$.

\begin{definition}[reduction sequence]
	A \tb{reduction sequence} of $P$  is a sequence of terms $t_0\to_P t_1\to_P t_2\to_P\ldots$. Write $t\downarrow s$ if there exists a reduction sequence $t=t_0\to_Pt_1\to_Pt_2\to_P\ldots\to_P t_l=s$.
\end{definition}

\begin{remark}
	Our definition of reduction is deterministic (at each step there is at most one possible reduction). It easily follows that for every $t$ there exists a unique reduction sequence starting with $t$. 
\end{remark}
	
\begin{lemma}\label{lemma: unicita}
	Given terms $t_0$, $t_1$, $t_2$, if $t_0 \downarrow t_1$ and $t_0 \downarrow t_2$ then $t_0 \to^*_P t_1 \to^*_P t_2$ or $t_0 \to^*_P t_2 \to^*_P t_1$ (This includes the case $t_1=t_2$). Additionally if $t_1$ and $t_2$ are natural numbers, then $t_1=t_2$. 
\end{lemma}
\begin{proof}
	By induction on the length of the reduction sequences and exploiting the fact that the reduction is deterministic. 
\end{proof}

\begin{definition}[termination]
	We say that $P$ \tb{terminates} on $\mathbf{u}$ if $f(\mathbf{u})\downarrow v$ for some natural number $v$, where $f$ is the entry function of $P$. We say that $P$ is \tb{terminating} if $P$ terminates on every input.  We also write $f(\mb u)=v$ for $f(\mb u)\ev v$ and $f(\mb u)\ev$ if there exists a $v$ such that $f(\mb u)\ev v$.
\end{definition}

\subsection{Size-change graphs}\label{subsection:sct}

We briefly recall the main definitions from \cite{FSY2017}.

A \tb{size-change graph} $G:f\to g$ is a finite set of constraints between the parameters of $f$ and  the parameters of $g$.  Constraints are of the form $x>y$ and $x\geq y$, where $x\in\var(f)$ and $y\in\var(g)$. Formally, we represent size-change graphs $G:f\to g$ as bipartite  graphs  with edges of the form $\arcd{x}{y}$ (a strict arc denoting the constraint $x>y$) and $\arce{x}{y}$ (a non-strict arc denoting the constraint $x\geq y$) with $x\in \var(f)$ and $y\in\var(g)$. We write $x\to y\in G$ as a shorthand for $\arcd{x}{y}\in G\lor \arce{x}{y}\in G$.

To $G:f\to g$ we can associate a transition relation $\xrightarrow{G}$ consisting of state transitions $(f,\mb u)\xrightarrow{G}(g,\mb v)$ for all  $\mb u, \mb v$ satisfying the given constraints. Note that $\xrightarrow{G}$ is decidable. Moreover, given $G_0:f\to g$ and $G_1:g\to h$, we can define the \tb{composition} $G_0;G_1:f\to h$ such that $\xrightarrow{G_0}\circ\xrightarrow{G_1}\subseteq\xrightarrow{G_0;G_1}$. 

The composition of two edges $\arce{x}{y}$ and $\arce{y}{z}$ is the edge $\arce{x}{z}$. In all other cases the composition of an edge, from $x$ to $y$, with an edge from $y$ to $z$ is the edge $\arcd{x}{z}$. The composition $G_0;G_1$ consists of all compositions of edges $x \to y \in G_0$ with edges $y \to z \in G_1$, with the exception of $\arce{x}{z}$ if $G_0;G_1$ contains $\arcd{x}{z}$. Formally: 
	\[
	\begin{split}
	E= \{\arcd{x}{z} : \ &\exists y \in \var(g)\  \exists r\in \bp{\downarrow,\Downarrow}((\arcd{x}{y}\in G_0 \wedge y\xrightarrow{r}z \in G_1) \\
	& \vee (x\xrightarrow{r}y\in G_0  \wedge \arcd{y}{z}\in G_1))\}\\
	\cup \{\arce{x}{z} : \ &\exists y \in \var(g)(\arce{x}{y}\in G_0 \wedge \arce{y}{z}\in G_1) \wedge \forall y \in \var(g) \\
	&\forall r, r'\in \bp{\downarrow,\Downarrow} ((x \xrightarrow{r} y \in G_0 \wedge y \xrightarrow{r'} z \in G_1) \implies r=r'=\mathord \Downarrow)\}.
	\end{split}
	\]

A \tb{description} $\mcal G$ of a program $P$ consists of size-change graphs $G_\tau:f\to g$ for any call $\tau:f\to g$ of $P$.   We say that $\mcal G$ is \tb{safe} if $\xrightarrow{\tau}\subseteq\xrightarrow{G_\tau}$ for all calls $\tau$ of $P$.

\begin{definition}
$\mcal G$ is \tb{MSCT} if for every infinite \tb{multipath} $M=G_0,G_1,\ldots$, i.e., an infinite sequence of size-change graphs with $G_i:f_i\to f_{i+1}$ and $G_{i+1}:f_{i+1}\to f_{i+2}$, contains an \tb{infinite descent}, i.e., a sequence of the form $x_t\to x_{t+1}\to x_{t+2}\to\ldots \to x_i\to\ldots$ with $x_i\in\var(f_i)$ such that for all $i$ we have $x_i\to x_{i+1}\in G_i$ and for infinitely many $i$ we have $\arcd{x_i}{x_{i+1}}\in G_i$. 
\end{definition}

Let $\cl(\mcal G)$ denote the closure of $\mcal G$ under composition.

\begin{definition}
$\mcal G$ is \tb{ISCT} if every \tb{idempotent} $G:f\to f$ in $\cl(\mcal G)$, i.e., $G^2=G$, contains a strict arc of the form $\arcd{x}{x}$ for some $x\in\var(f)$. 
\end{definition}

\subsection{P\'{e}ter-Ackermann}\label{Section:Ackermann}\label{genP-A} 

As Ben-Amram shows in \cite{BenAmram}, the  P\'{e}ter-Ackermann function is ISCT. 
\[
\begin{split}
A(x,y) = &\mbb{ if } x=0 \mbb{ then } y+1\\
&\mbb{ else if } y=0 \mbb{ then } \tau_0: A(x-1,1)\\
&\qquad \mbb{ else } \tau_1: A(x-1, \tau_2: A(x,y-1))\\
\end{split}
\]
Note that we have three calls $\tau_i$ ($i<3$) which are safely described by the following size-change graphs:
	\begin{center}
		
		\begin{minipage}[c]{0.3\textwidth}
			\[\begin{tikzpicture}[xscale=1,yscale=1]
			
			\node (x1) at (0, 1) {$x$};
			\node (y1) at (0, 0) {$y$};
			\node (x2) at (2, 1) {$x$};
			\node (y2) at (2, 0) {$y$};
			\path (x1) edge [->]node [auto] {${\downarrow} $} (x2);
			\node (G) at (1,-1) {$G_{1}:A\to A$};
			\end{tikzpicture}
			\]
			
		\end{minipage}
		\begin{minipage}[c]{0.3\textwidth}
			\[\begin{tikzpicture}[xscale=1,yscale=1]
			
			\node (x1) at (0, 1) {$x$};
			\node (y1) at (0, 0) {$y$};
			\node (x2) at (2, 1) {$x$};
			\node (y2) at (2, 0) {$y$};
			\path (x1) edge [->]node [auto] {${\Downarrow} $} (x2);
			\path (y1) edge [->]node [auto] {${\downarrow} $} (y2);
			\node (H) at (1,-1) {$G_2:A\to A$};
			\end{tikzpicture}
			\]
		\end{minipage}
	\end{center}
	
	The size-change graph $G_{1}$ safely describes both calls  $\tau_0 : A(x-1,1)$ and $\tau_1: A(x-1,A(x,y-1))$. In particular, notice that in the call $\tau_1$ the parameter value $x$ decreases no matter what the value of the expression $A(x,y-1)$ is. Finally, the size-change graph $G_2$ safely describes the call $\tau_2: A(x,y-1)$. \smallskip

Actually, we can prove that P\'{e}ter-Ackermann is MSCT within $\RCAo$. 

\begin{lemma}[$\RCAo$]\label{lemma: AckSCT}
	$A$ is MSCT.
\end{lemma}
\begin{proof}
 Every multipath contains either infinitely many $G_1$ or cofinitely many $G_2$. In the first case we have an infinite descent in the parameter $x$ and in the second case we have an infinite descent in the parameter $y$.
\end{proof}

These remarks highlight that both ISCT Soundness and MSCT Soundness are not provable over $\RCA$. 

\section{The standard proof requires $\ACA$}

We discuss the standard proof of SCT soundness. This section is self-contained.

\begin{theorem}[Lee, Jones, Ben-Amram]
	If $P$ is MSCT then $P$ is terminating.
\end{theorem}
\begin{proof}[Proof sketch]
	Let $\mcal G$ be a safe description of a program $P$ and suppose that $P$ does not terminate on $\mathbf{u}$. Then there exists an infinite sequence of state transitions $(f,\mathbf{u})\xrightarrow{\tau_0}(f_1,\mathbf{u_1})\xrightarrow{\tau_1}(f_2,\mathbf{u_2})\ldots$. Consider the corresponding multipath in $\mathcal{G}$. As $\mcal G$ is MSCT, there exists an infinite descent. By safety, we have an infinite descending sequence of natural numbers. A contradiction. 
\end{proof}

We will show that this proof requires $\ACA$. The key step in the proof is the existence of an infinite state transition sequence, i.e.,  the existence of an infinite branch in the so called activation tree.

\begin{definition}[activation tree]\label{activation}
	Let $P$ be a program with entry function $f$. Given $\mathbf{u}\in\NN^{\arity(f)}$, the \tb{activation tree} $T_P^{\mathbf{u}}$ consists of all finite state transition sequences starting with $(f,\mb u)$, i.e., sequences of the form:
	\[  (f,\mathbf{u})\xrightarrow{\tau_0}(f_1,\mathbf{u}_1)\xrightarrow{\tau_2}\ldots \xrightarrow{\tau_{k-1}}(f_k,\mathbf{u}_k).\]
\end{definition}

The activation tree is $\Sigma^0_1$ and finitely branching. If $P$ is MSCT, then every branch of the tree is finite (since we are considering values in $\NN$, they cannot decrease infinitely many times). Therefore, we have:

\begin{proposition}[$\RCAo$]
	If $P$ is MSCT, then $T_P^{\mathbf{u}}$ has no infinite branches for all $\mathbf{u}$.
\end{proposition}
\begin{proof}
	From the definitions. The reader is encouraged to fill in the details. 
\end{proof}

One can show that $P$ terminates iff $T_P^{\mb u}$ is finite for all $\mb u\in\NN^{\arity(f)}$. With some effort, this can be done in $\RCA$. 

\begin{proposition}[$\RCAo$]\label{termination tree}
 $P$ terminates on $\mathbf{u}$ iff  $T_P^{\mathbf{u}}$ is finite. 
\end{proposition}

\begin{proof}
	Let  $T=T_P^{\mathbf{u}}$. Suppose first that $T$ is finite. By bounded $\Sigma^0_1$-comprehension the tree $T$ exists. By $\Sigma^0_1$-induction on $T$ we show that for all $\sigma\in T$, if $\last(\sigma)=(g,\mb v)$, then $g(\mb v)\ev$. Write $t\nev$ if there is no natural number $v$ such that $t\ev v$. Note that if $t\nev$ then there exists a subterm  $h(\mb{s})$, with $h\in\fun$, and $\mb{w}\in\N^n$, such that $\mb{s}\ev\mb{w}$ and $h(\mb{w})\nev$.  This can be proved by $\Pi^0_1$-induction on a given term $t$ such that $t\nev$. In fact, the least subterm $s$ of $t$ (in any linear ordering of subterms of $t$ which respects the subterm relation), such that $s\nev$, is as desired. Suppose $\sigma$ is an end-node. If $g(\mb v)\nev$ then $e^g(\mb v)\nev$, so there exists a subterm $h(\mb s)$ as above. Then $(g,\mb v)\xrightarrow{\tau}(h,\mb w)$ for some $\tau$, so $\sigma$ is not an end-node. The case when $\sigma$ is not an end-node can be proved similarly. 
	
	In the other direction, suppose that $P$ terminates on $\mb u$ with witness $f(\mb u)=t_0\to_P t_1\to_P\ldots\to_Pt_{l}=v\in\N$. Let $\mcal T$ be the set of all subterms appearing in the reduction sequence. We claim that if $g(\mb v)\in\mcal T$ and $(g,\mb v)\xrightarrow{\tau} (h, \mb w)$, then also $h(\mb w)\in\mcal T$. Let $g(\mb v)$, with $e^g(\mb v)|_\tau=h(\mb t)$ and $\mb t_i \to^*_P \mb w_i$, be given. Take the maximum $i\leq l$ such that $g(\mb v)$ appears in $t_i$. Then $t_{i+1}$ is obtained by reducing $g(\mb v)$, 
	so $h(\mb t)$ is a subterm of $t_{i+1}$. Since the program terminates, there exists a tuple of natural numbers $\mb n$ such that $\mb t_i \to^*_P \mb n_i$. By Lemma \ref{lemma: unicita}, we have $\mb w = \mb n$. Since, moreover, the reduction is deterministic, $h(\mb w)$ must be a subterm of $t_j$ for some $i<j\leq l$, hence $h(\mb w)\in\mcal T$. This proves the claim. Now, given $\sigma\in T$, one can show by induction that every initial segment of $\sigma$ consists of pairs $(g,\mb v)$ with $g(\mb v)\in \mcal T$. It easily follows that $T$ is finite. 
\end{proof}


\begin{proposition}[$\ACA$]\label{ACA}
	If $P$ does not terminate on $\mathbf{u}$, then $T_P^{\mathbf{u}}$ has an infinite branch. 
\end{proposition}
\begin{proof}
Suppose that $P$ does not terminate on $\mathbf{u}$. Then $T=T_P^{\mathbf{u}}$ is infinite by Proposition \ref{termination tree}. Note that the tree $T$ exists within $\ACA$. Since $T$ is finitely branching, it has an infinite branch by K\"{o}nig's lemma (which is provable in $\ACA$). 
\end{proof}

\begin{proposition}[$\RCA$]
	The statement ``If $P$ does not terminate on $\mathbf{u}$, then $T_P^{\mathbf{u}}$ has an infinite branch'' implies $\ACA$.
\end{proposition}
\begin{proof}
	We argue in $\RCA$. By \cite[Lemma III.1.3]{SOSOA}, it is enough to show that the range of any one-to-one function exists. Let $f\colon\N\to\N$ be given. Say that $x\in\N$ is an $f$-true stage (or simply true) if $f(x)<f(y)$ for all $x<y$. It is well-known that, provably in $\RCA$, we can define the range of $f$ from any infinite set of $f$-true stages (see, e.g, \cite{Fri16}). We show that the set of true stages exists.  We may safely assume that $0$ is true.

	Define $P$ as follows:
	\[\begin{split}
	g(x,t,y,z)= &  \mboxb{ if }  f(z)<f(y) \mboxb{ then } z \\
	& \mboxb{ else } g(x+1,\bot,y,g(x+1,\top, x+1,x+2)) \\
	\end{split} \]
	The idea is that $g(x,t,y,z)$ tests the truth of $y$ by seeking the least  $v\geq z$ such that  $f(v)<f(y)$.

\begin{claim}
If $y\leq x<z$ and $g(x,t,y,z)=v$ then $z\leq v$ and $f(v)<f(y)$.
\end{claim}
\begin{proof}
By $\Pi^0_1$-induction on the length of a reduction sequence.	
\end{proof}

	By the claim, $P$ does not terminate on $\mb u= (0,\top,0,1)$. Let $(g,\mb u)\to (g,\mb u_1)\to\ldots (g,\mb u_x)\to $ be an infinite branch in $T_P^{\mb u}$, where $\mb u_x=(x,t_x,y_x,z_x)$. By  $\Sigma^0_0$-induction it is easy to show that $y_x\leq y_{x+1}$ and $y_x\leq x<z_x$ for all $x$. 
	
We now show that for all $x\in\N$,  
\[ x \text{ is true if and only if }  t_x=\top, \]   
so we can define the set of true stages by $\Delta^0_1$-comprehension (indeed $\Delta^0_0$). 

Suppose $x$ is true and $t_x=\bot$. Then  $\mb u_x=(x,\bot, y,g(x,\top,x,x+1))$ and by the claim above we have that $f(z_x)<f(x)$ with $x<z_x$, so $x$ is not true, a contradiction. 

Suppose that $t_x=\top$ and $x$ is false. We have that $y_x=x$.  Let $v>x$ be least such that $f(v+1)< f(x)$. Consider $\mb u_v=(v,t,y,z)$. Now $\mb u_v=(v,\top, v, v+1)$ or $\mb u_v=(v,\bot, y, v+1)$ with $x\leq y\leq v$. By the minimality of $v$ we have $f(z)<f(y)$ in both cases. Thus there is no state transition from $\mb u_v$, a contradiction. 
\end{proof}

\begin{remark}
	Within our syntax, primitive operators do not appear in boolean expressions. We can modify $P$ as follows:
	
	\[\begin{split}
	g(x,t,y,z)&=h(f(y),f(z),z,g(x+1,\bot,y,g(x+1,\top,x+1,x+2))) \\
	h(a,b,c,d)&= \mboxb{ if }  b < a \mboxb{ then } c \mboxb{ else } d
	\end{split}\]
	This program computes the same function. Observe that this program does not have a safe SCT description.
\end{remark}


\section{Lower bound}

In this section we show that ISCT soundness implies $\WO(\omega_3)$ over $\RCA$. To this end, we  consider the (relativized) fast growing hierarchy. 

\subsection{Fast Growing Hierarchy} 
We formalise $\varepsilon_0$ in $\RCAo$ as in Definition 2.3 from \cite{simpsonord}:
\begin{definition}
The set $\mathcal{E}$ of notations of ordinals $<\varepsilon_0$ and order $<$ on $\mathcal{E}$ is taken as follows:
\begin{enumerate}
  \item If $\alpha_0 \geq \dots \geq \alpha_n \in \mathcal{E}$, then $\omega^{\alpha_0} + \dots + \omega^{\alpha_n} \in \mathcal{E}$.
  \item $\omega^{\alpha_0} + \dots + \omega^{\alpha_n} < \omega^{\beta_0} + \dots + \omega^{\beta_m}$ if and only if:
  \begin{enumerate}
    \item $n<m$ and  $\alpha_i=\beta_i$ for all $i \leq n$, or:
    \item there is $i\leq \min\{ n,m\}$ with $\alpha_j=\beta_j$ for all $j < i$ and $\alpha_i < \beta_i$.
  \end{enumerate}
\end{enumerate}
We use $0$ to denote the empty sum, $0 < \alpha$ for all $\alpha \neq 0$, $1=\omega^0$, $n=\overbrace{1 + \dots +1}^{n}$, $\omega=\omega^1$, $\omega_0(\alpha)=\alpha$, $\omega_{d+1}(\alpha)=\omega^{\omega_d(\alpha)}$ and $\omega_d=\omega_d(1)$. 
\end{definition}

\begin{remark}
To show that $\mathcal{E}$ is well defined in $\RCAo$, one need only observe that the corresponding characteristic function is primitive recursive.
\end{remark}

\begin{definition}[fundamental sequence]
For $\alpha = \omega^{\alpha_0} + \dots + \omega^{\alpha_n} \in \mathcal{E}$ and $x \in \N$, take $0[x]=0$, $(\alpha+1)[x]=\alpha$, and for $\alpha$ limit ordinal:
\begin{enumerate}
  \item If $\alpha_n=\beta+1$, then $\alpha[x] = \omega^{\alpha_0} + \dots + \omega^{\alpha_{n-1}} + \omega^\beta \cdot x$,
  \item If $\alpha_n$ is a limit, then $\alpha[x] = \omega^{\alpha_0} + \dots + \omega^{\alpha_{n}[x]}$.
\end{enumerate}
\end{definition}

For well-ordered $\alpha<\varepsilon_0$, the fast growing hierarchy relative to $f:\NN\to\N$ would be defined as follows:
\begin{align*}
F_{0,f}(x) &= f(x),  \\
F_{\alpha+1,f}(x) & = F_{\alpha,f}^{(x+1)}(1),  \\
F_{\lambda,f}(x)  & =F_{\lambda[x],f}(x)  \text{ if $\gamma$ is limit},
\end{align*}

where $F^{(n)}$ is the $n$-times iterate of a function  $F$, defined by $F^{(0)}(x)=x$ and $F^{(n+1)}(x)=F(F^{(n)}(x))$. 

\begin{remark}
In the usual definition of the fast growing hierarchy we have $F_{\alpha+1}(x)=F_\alpha^{(x+1)}(x)$. For our purposes we use this slightly modified version (see Proposition \ref{FA}). 
\end{remark}

\begin{remark}
The primitive recursive functions are exactly the functions elementary recursive in some $F_n$ with $n<\omega$. Let $\mc F_\alpha$ be the set of functions elementary recursive in $F_\alpha$.  The multiply recursive functions (functions defined by transfinite recursion on $\omega^n$ for some $n\in\NN$)  are exactly the functions in $\bigcup_{\alpha<\omega^\omega} \mc F_\alpha$. Ben-Amram proved that SCT programs compute exactly the multiply recursive functions. 
\end{remark}

We follow standard practice  in defining the fast growing hierarchy in terms of its canonical computation. Define the following function $K_f\colon (\varepsilon_0)^* \times \N \rightarrow (\varepsilon_0)^* \times \N$. Intuitively, this function represents one step in the computation of $F_{\alpha,f}(x)$. Let
\[
K_f( \alpha_0 \dots \alpha_n,  x  )= \left\{ 
\begin{array}{ll}
(\alpha_0 \dots \alpha_{n-1}, f(x)) 													& \textrm{if $\alpha_n=0$} \\
(\alpha_0 \dots \alpha_{n-1}\overbrace{\beta \dots \beta}^{\textrm{$x+1$ times}}, 1)	& \textrm{if $\alpha_n=\beta+1$} \\
(\alpha_0 \dots \alpha_{n-1} \alpha_n [x], x)											& \textrm{if $\alpha_n$ is a limit} \\
\end{array}
\right.
\]
and $K_f( \langle \rangle , x) =(\langle \rangle, x)$. Notice that $(\alpha_0 \dots \alpha_n, x)$ simply represents the term \[F_{\alpha_0} ( \dots ( F_{\alpha_n} (x) \dots ). \] $F_f$ is the result of repeated applications of the `computation steps' (when it exists).
\begin{definition} 
Let	$F_{\alpha,f} (x) =\mu y. \exists l \ K_f^{(l)} (\alpha , x)= ( \langle \rangle , y )$. We call the sequence $\{K_f^{(i)} (\alpha , x)\}_{i\in \N}$  the \tb{computation} of $F_{\alpha,f}(x)$. Say that the computation is finite if there exists $l$ such that $K_f^{(l)} (\alpha , x)= ( \langle \rangle , y )$. 
\end{definition}
One can show that this definition is equivalent to usual $\Delta^0_1$ definitions as in, e.g. \cite{HP93} (adapted to take into account the different initial function and slightly different conditions).

\begin{proposition}[$\RCAo$]\label{wo tot}
	For all $\alpha<\varepsilon_0$, \[\WO(\omega^\alpha)\biimp (\forall f:\NN\to\N)\Tot(F_{\alpha,f}). \]
\end{proposition}

\begin{proof} Take $h(\alpha_0 \dots \alpha_n, x)=\omega^{\alpha_0} + \dots + \omega^{\alpha_n}$ and $h(\langle \rangle , x)=0$. Note that the computation of $F_{\alpha,f}(x)$ is finite iff $\beta_i=0$ for some $i$, where $\beta_i=h(K_f^{(i)} (\alpha , x))$. Note also that $\beta_i>\beta_{i+1}$ as long as $\beta_i>0$.

First suppose that $\omega^\alpha$ is well-founded. Then the sequence $\{ h( K_f^{(i)} (\alpha , x))\}_{i \in\N}$ reaches zero, so the computation is finite. 

For the other direction, we assume that $\omega^\alpha$ is not well-founded and prove that $F_{\alpha,f}$ is not total. First, a definition.

\begin{definition}[Maximal coefficient]
	By primitive recursion on the construction of $\alpha<\varepsilon_0$, define $\mathrm{mc}(\alpha)\in\NN$ as follows. Let
	$\mathrm{mc}(0)=0$.  If $\alpha=\omega^{\alpha_0}\cdot a_0 + \dots + \omega^{\alpha_n}\cdot a_n$, where  $\alpha_0 > \dots > \alpha_n$ and  $a_i >0$, then 
	\[
	\mathrm{mc}(\alpha)=\max\{ \mathrm{mc}(\alpha_i), a_i \}.
	\]
\end{definition}

Given an infinite descending sequence \[ \omega^\alpha>\alpha_0>\alpha_1> \alpha_2> \dots, \]
take $f(x)>\mathrm{mc}(\alpha_{x+1})+x+1$ and strictly increasing. Assume, for a contradiction, that that the computation of $F_{\alpha,f}(f(0))$ is finite. 

To ease notation, take
\[
(\sigma_i,x_i)=K^{(i)}_f(\alpha, f(0)) \text{ and } \beta_i=h(\sigma_i,x_i).
\] 
We aim to show that $\beta_i>0$ for all $i$, in contradiction with the finiteness of the computation of $F_{\alpha,f}(f(0))$.  Note that $\beta_0=\omega^\alpha$ and $x_0=f(0)$. One can show that:
\begin{enumerate}[(1)]
	\item if $\gamma > \beta$ is a limit, then $\gamma[\mathrm{mc}(\beta)+1] > \beta$;
	\item $F_{\beta,f} (y) \geq f(y)$, hence $F_{\beta,f}^{(y)} (1)> y$ for all $\beta\leq\alpha$, $y$ which occur in the computation of $F_{\alpha,f}(f(0))$;
	\item  if $l$ is the smallest such that $K_f^{(l)} (\beta , y)= ( \langle \rangle , z )$, then $K_f^{(l)} (\sigma\beta , y)= ( \sigma , z )$.
\end{enumerate}

By primitive recursion let us define an increasing sequence $(a_i)$ of natural numbers as follows.
Set $a_0=0$. If  $\sigma_{a_i}$ ends with a zero or $\sigma_{a_i}=\langle \rangle$, let $a_{i+1}=a_i+1$. Otherwise,  let $a\geq a_i$ be the smallest such that $\sigma_a$ ends with a successor $\beta+1$, and set $a_{i+1}=a+l+1$, where $l$ is the least such that \[ K_f^{(l)}(\overbrace{\beta \dots \beta}^{i+1} , 1)=(\langle\rangle, z). \]

\begin{claim} For every $i>0$ we have:
\[
\beta_{a_i} > \alpha_i
\]
and
\[
x_{a_i} \geq f(i).
\]
\end{claim}
From the claim, it follows that $\beta_i>0$ since $\beta_i\geq\beta_{a_i}>\alpha_i>0$, as desired. 

\noindent\emph{Proof of the claim}. Induction on $i$. For $i=0$, the claim follows directly. For the induction step, assume that the claim is true for $i$.

Case 1. $a_{i+1}=a_i+1$. Since $\beta_i > 0$, the inequalities follow directly from the definitions:
\[
\beta_{a_i+1}=\beta_{a_i}-1\geq \alpha_{i}>\alpha_{i+1}
\]
and 
\[
x_{a_i+1}= f(x_{a_i}) \geq f(f(i)) \geq f(i+1).
\]

Case 2. Let $a$ and $\beta$ be those from the definition of $a_{i+1}$. By the definition of $a$,  $\sigma_j$ ends with a limit for  all $j \in [a_i, a)$. Therefore, by the induction hypothesis and notice (1), $\beta_a > \alpha_{i}$ and $x_a \geq f(i)$.  Since $\sigma_a$ is of the form  $\gamma_0 \dots \gamma_l \beta+1$, $\sigma_{a+1}$ has the form:
\[
\gamma_1 \dots \gamma_l \overbrace{\beta \dots \beta}^{\geq \mathrm{mc}(\alpha_{i+1})+1} \overbrace{\beta \dots \beta}^{i+1},
\]
so $\beta_{a_{i+1}}\geq \beta_{a}[\mathrm{mc}(\alpha_{i+1})+1] > \alpha_{i+1}$ by notice (1) and (3). By notice (2) and (3),  $x_{a_{i+1}} \geq F_{\beta,f}^{(i+1)}(1) \geq f(i+1)$. This ends the proof of the claim.
\end{proof}

\subsection{Generalizing P\'{e}ter-Ackermann}

Recall that for  $f\colon\N\to\N$ and $\alpha<\varepsilon_0$  we have:

\begin{align*}
F_{0,f}(x) & = f(x)  \\
F_{\alpha+1,f}(x) & = F_{\alpha,f}^{(x+1)}(1) \\
F_{\lambda,f}(x) &=F_{\lambda[x],f}(x)
\end{align*}
Note that $F_{x,f}(y)=A_f(x,y)$, where $A_f$ is the P\'{e}ter-Ackermann function relativized to $f$ (see \cite{KK}). We now generalize $A_f(x,y)$ as follows.  For $n>0$ and $f\colon\N\to\N$, let
\[\begin{split}
A^{n}_f(x_1,x_2,\ldots,x_n,y) = &\mboxb{ if } x_1=\ldots=x_{n}=0 \mboxb{ then } f(y)\\
&\mboxb{ else if } x_1>0 \land x_2=\ldots=x_{n}=0  \mboxb{ then } \\
& \ \ \ \ \tau_1: A^{n}_f(x_1-1,y,x_3,\ldots,x_n,y)  \\
& \ \  \vdots \\
& \mboxb{ else if } x_{i}>0 \land x_{i+1}=\ldots=x_n=0  \mboxb{ then }  \\ 
& \ \ \ \ \tau_i: A^{n}_f(x_1\ldots,x_{i}-1,y,x_{i+2},\ldots,x_{n},y) \\
& \ \  \vdots \\
&\mboxb{ else if } x_{n-1}>0 \land x_n=0 \mboxb{ then } \\
& \ \ \ \ \tau_{n-1}: A^{n}_f(x_1,\ldots,x_{n-1}-1,y,y) \\
&\mboxb{ else if } x_n>0 \land y=0 \mboxb{ then } \\
& \ \ \ \ \tau_n: A^{n}_f(x_1,x_2,\ldots,x_{n}-1,1) \\
& \mboxb{ else } \\
& \ \ \ \ \tau_0: A^{n}_f(x_1,\ldots,x_{n-1},x_n-1,\tau_{n+1}: A^{n}_f(x_1,\ldots,x_{n},y-1)) \\
\end{split}
\]

In the interest of readability, let $\mb x=x_1,\ldots,x_n$ and  $\alpha(\mb x)=\omega^{n-1}x_{1}+\ldots+x_n$. For $x_n>0$, let  $\mb x-1=x_1,\ldots,x_{n-1},x_n-1$ and observe that $\alpha(\mb x-1)=\alpha(\mb x)-1$. Then: 

\[ A^n_f(\mb x, y)=\begin{cases} f(y) & \text{ if } \alpha(\mb x)=0 \\
            A^n_f(\mb x', y) & \text { if } \alpha(\mb x) \text{ is limit and } \alpha(\mb x')=\alpha(\mb x)[y] \\
            A^n_f(\mb x-1,1) & \text{ if } \alpha(\mb x) \text{ is successor and } y=0 \\
            A^n_f(\mb x-1, A^n_f(\mb x,y-1)) & \text{ if } \alpha(\mb x)   \text{ is successor and }  y>0    

\end{cases}\]

We now show the relationship between the fast growing hierarchy and the generalized  P\'{e}ter-Ackermann function.

\begin{proposition}[$\RCA$]\label{FA}
	For all $n>0$ and $f\colon\N\to\N$,
	\[        A^{n}_f(\mb x,y)=F_{\alpha(\mb x),f}(y).  \]
	That is, for all $\mathbf{x},y,z$, $ A^{n}_f(\mb x,y)= z$  iff there exists $l$ such that $K_f^{(l)}(\alpha(\mb x),y)=(\langle\rangle,z)$.
\end{proposition}

\begin{proof}
	
	Write $A$ for $A^n_f$ and $\alpha$ for $\alpha(\mb x)$. Say that $A(\mb x,y)=z$ in $l$-many steps if there exists a reduction sequence $A(\mb x,y)=t_0\to t_1\ldots \to t_l=z$. We also write $A(\mb x,y)\to^{(l)} z$.  On the other hand,  $K_f^{(l)}(\alpha,y)=(\langle\rangle, z)$ iff there exists a sequence  $(\alpha,y)=k_0\to k_1\to\ldots\to k_l=(\langle\rangle, z)$ where $k_{i+1}=K_f(k_i)$. 
	We also write $k_0\to^{(l)} k_l$.  We shall use the fact that $(\tau,x)\to^{(l)}(\rho, y)$ iff  $(\sigma\tau,x)\to^{(l)}(\sigma\rho,y)$ for all $\sigma$. 
	
	In one direction, we prove by $\Pi^0_1$-induction on $l$ that, for all $\mathbf{x},y,z$, if $A(\mathbf{x},y)= z$ in $l$-many steps then $(\alpha,y)\to^{(l)}(\langle\rangle, z)$. This is a relatively straightforward, if tedious, verification. Let us consider the case $x_n,y>0$. The other cases are similar and actually simpler.  Since $x_n>0$ we have that  $\alpha=\alpha(\mb x)$ is a successor. Note that $\alpha-1=\alpha(\mb x-1)$, where $\mb x-1=x_1,\ldots,x_{n-1},x_n-1$.  Let
	\[  A(\mathbf{x},y)\to A(\mb x-1, A(\mb x ,y-1))\to^{(l_0)} A(\mb x-1, y')\to^{(l_1)} z \]
	with $l=l_0+l_1+1$.  Then   $A(\mb x ,y-1)= y'$ in  $l_0$-many steps  and  $A(\mb x-1,y')= z$ in $l_1$-many steps.  By the induction hypothesis,  $(\alpha,y-1)\to^{(l_0)}(\langle\rangle,y')$ and
		 $(\alpha-1,y')\to^{(l_1)}(\langle\rangle, z)$.  Since $(\alpha,y-1)\to((\alpha-1)^{(y)},1)$, it follows that $((\alpha-1)^{(y)},1)\to^{(l_0)}(\langle\rangle,y')$. 
	Therefore we have:
		\[ 
		(\alpha,y)\to(\alpha-1(\alpha-1)^{(y)},1)\to^{(l_0)} (\alpha-1,y')\to^{(l_1)} (\langle\rangle,z)
		\]
		with $l_0+l_1+1=l$. 
		
\smallskip 

For the other direction, we show by $\Pi^0_1$-induction on $l$ that,  for all $\mathbf{x},y,z$, if $(\alpha,y)\to^{(l)}(\langle\rangle, z)$, then  $A(\mb x, y)=z$ in less than $l^2$-many steps. The $l^2$ bound is not optimal but does the job.  Once again, consider the case $x_n,y>0$ so that $\alpha=\alpha(\mb x)$ is a successor and $\alpha-1=\alpha(\mb x-1)$ with $\mb x-1=x_1,\ldots,x_{n-1},x_n-1$.  Suppose that 
\[ (\alpha,y)\to((\alpha-1)^{(y+1)},1)\to^{(l_0)} (\alpha-1,y')\to^{(l_1)} (\langle\rangle,z). \]
Then $l=l_0+l_1+1$. As before, note that $(\alpha,y-1)\to((\alpha-1)^{(y)},1)$. By induction, $A(\mb x,y-1)=y'$ within $\leq (l_0+1)^2$-many steps and  $A(\mb x-1,y')=z$ within $\leq l_1^2$-many steps, where $\mb x-1=x_1,\ldots,x_{n-1},x_n-1$. Therefore we have a reduction sequence 
\[   A(\mb x,y)\to A(\mb x',A(\mb x,y-1))\to\ldots\to A(\mb x',y')\to\ldots\to z \]
of length $\leq (l_0+1)^2+l_1^2+1\leq l^2$. Note in fact that $l_i>0$. 
\end{proof}

\begin{corollary}[$\RCA$]\label{wo tot A} The following holds:
	\begin{itemize}
		\item $\WO(\omega^{\omega^\omega})\biimp (\forall n>0)(\forall f:\NN\to\NN) \Tot(A^n_f)$
		\item $\WO(\omega^{\omega})\biimp (\forall f:\NN\to\N)\Tot(A^{2}_f)$
	\end{itemize}
\end{corollary}
\begin{proof}
	This follows from Proposition \ref{wo tot} and Proposition \ref{FA}.
\end{proof}

\subsection{From soundness to $\WO(\omega_3)$} 
We can now give the desired lower bounds. 

\begin{definition}[description $\mathcal{A}_n$ of $A^n_f$] It is convenient to define $\mcal A_n$ on parameters $x_1,\ldots,x_n,x_{n+1}$. That is, we write $x_{n+1}$ for $y$.  Define $\mathcal{A}_n=\{A_1,\ldots,A_{n+1}\}$ as follows.  For every $0<j\leq n+1$, let $A_j$ be the size-change graph with arcs $\arcd{x_j}{x_j}$ and $\arce{x_i}{x_i}$ for all $0<i<j$. 
\end{definition}

Note that $\mathcal{A}_n$ does not depend on $f$. 

\begin{proposition}[$\RCA$]
For all $n\in\N$ and $f\colon \N\to\N$, $\mathcal{A}_n$ is a safe ISCT description of $A^n_f$. More precisely, $A_i$ is a safe description of $\tau_i$ for all $0<i\leq n+1$, and $A_n$ is a safe description of $\tau_0$. 
\end{proposition}
\begin{proof}
	It is easy to see that $\mathcal{A}_n$ is a safe description. Let us show that every $G\in{\sf cl}(\mathcal{A}_n)$ has an arc $\arcd{x}{x}$. Let $G=G_0;G_1;\ldots;G_{k-1}$ with  $G_j\in\mathcal{A}_n$ for every $j<k$. Let $0<i\leq n+1$ be least such that $A_i\in\{G_0,\ldots,G_{k-1}\}$. Then  $\arcd{x_i}{x_i}\in G$. 
\end{proof}

Note that the size-change graphs defined in this description could be extended to other size-change graphs, which also safely describe $P$, by adding to $A_j$ the arcs $\arce{y}{x_{j+1}}$, $\arce{y}{y}$ and $\arce{x_{i}}{x_{i}}$ for every $j< i \leq n$. Anyway for our goals the description above is more suitable. 

\begin{corollary}[$\RCA$]\label{Corollary: lower bound}
	ISCT soundness implies $\WO(\omega_3)$.
\end{corollary}
\begin{proof}
This follows from Corollary \ref{wo tot A}.
\end{proof}

\begin{remark}\label{Remark: fromMSCTtoAck}
	Note that for any standard $n>0$, $\RCA$ proves that $\mathcal{A}_n\text{ is MSCT}$. In particular, MSCT soundness implies $\WO(\omega^{\omega})$ by Corollary \ref{wo tot A}.  It turns out that proving MSCT for all $\mcal A_n$ requires $\is2$.  
	\end{remark}

\begin{proposition}[$\RCA$]\label{Propositon:MSCTandIS2}
	The following are equivalent:
	\begin{itemize}
		\item $\is2$;
		\item For all $n>0$, $\mathcal{A}_n$ is MSCT.
	\end{itemize}
\end{proposition}
\begin{proof}
	For the forward direction, let $M=G_0,G_1,\ldots$ be a multipath with $G_i\in\mcal{A}_n$. Let $0<i\leq n+1$ be least such that $A_i$ appears infinitely often. Then there is an infinite descent starting with $x_i$.
	
	For the reversal, one can adapt the proof of \cite[Theorem 6]{FSY2017}. For the sake of completeness, we briefly describe the main idea. Starting point is that $\is2$ is equivalent to the Strong Pigeonhole Principle (see e.g., \cite{FSY2017}) which states that given a coloring in $k$ many colors of the natural numbers, there exists the set of colors which appear infinitely many times in this coloring. Therefore it is sufficient to show that for  every finite coloring $c:\NN\to k$ the set $I^\infty=\{i<k\colon (\exists^\infty x)c(x)=i\}$ exists. As in the proof of \cite[Theorem 6]{FSY2017} we can define for all $x\in\NN$ the set $\mcal I_x$ of guesses at stage $x$. That is, every $I\in\mcal I_x$ is a non-empty subset of $k$ and $I\subseteq I^\infty$ iff $I\in\mcal I_x$ for infinitely many $x$. Let $n+1=2^k-1$. Then we have $n+1$-many non-empty subsets of $k$, say $I_1,\ldots,I_{n+1}$. We can assume that $|I_i|<|I_j|$ implies $i>j$. Now define a multipath $M=G_0,G_1,\ldots$ in $\mcal A_n$ by letting $G_x=A_i$, where $i$ is least such that $I_i$ is a guess at stage $x$ of maximal size. By the assumption there exists an infinite descent starting from some parameter $x_i$ with $0<i\leq n+1$ at some point $t$. We claim that $I_i=I^\infty$. Since there are infinitely many arcs of the form $\arcd{x_i}{x_i}$, we have that $I_i\subseteq I^\infty$. Now suppose for a contradiction that $I_i\neq I^\infty$. Then there exists a stage $x>t$ with a guess $I$ of size bigger than $I_i$. Therefore, by definition, there exists $j<i$ such that $G_x=A_j$ and so in $G_x$ there is no arc from $x_i$ to $x_i$, a contradiction. 
\end{proof}


\section{Upper bound}

In this section we aim to show the following:

\begin{theorem}[$\RCA$]\label{upper bound}
	$\WO(\omega_3)$ implies ISCT soundness.
\end{theorem}

From the proof of this result we then extract a bound for the length of computations of tail-recursive ISCT programs (see Subsection \ref{tail-recursive}). 
\begin{definition}
	We call \tb{proper} a term of the form $f(\mb u)$ with $f\in\fun$ and $u\in\N^{\arity(f)}$. We also say that a reduction $t\to_P s=t[f(\mb u)]_\rho$ is \tb{proper} if $f\in\fun$.
\end{definition}

\tb{Proof idea}: Given an infinite reduction sequence \[ f(\mb u)=t_0\to_P t_1\to_P \ldots \to_P t_n\to_P\ldots \] we assign ordinals \[\omega_3>\alpha_0>\alpha_1>\ldots>\alpha_n>\ldots \]  

Actually, since we also consider reductions involving primitive operators,  we will define a non-increasing sequence of ordinals which decreases infinitely many times. For the details, one can go to the proof. Here we just outline some ideas that hopefully will make the proof easier to follow.  

As in Tait \cite{Tait},  we want to assign to each term $t_n$ a finite set of ordinals $\gamma_0,\gamma_1,\ldots$ in $\omega^\omega$. Each $\gamma_i$ corresponds to a subterm of $t_n$ of the form $g(\mb s)$ with $g\in\fun$. 
To $t_n$ we assign the ordinal \[\alpha_n=\bigoplus_{i}\beta_i<\omega_3, \] where 
\[ \beta_i=\begin{cases} \omega^{\gamma_i}\cdot a & \text{ if $g(\mb s)$ is proper } \\ \omega^{\gamma_i} & \text{ otherwise} \end{cases} \]

Here $a$ is the the maximum \tb{size} of $e^f$ for $f\in\fun$. The size of an expression is defined as the number of function symbols. The idea is that a one-step reduction gives rise to at most $a$-many subterms of the form $g(\mb s)$ with $g\in\fun$. We use the following basic fact from ordinal arithmetic without further notice: if $\gamma>\gamma_0,\gamma_1,\ldots$ then $\omega^\gamma>\bigoplus_{i}\omega^{\gamma_i}$. Therefore, if a term $t$ with assigned ordinal $\gamma$ gives rise to $a$-many subterms $s_0,s_1,\ldots$ with ordinals $\gamma_0,\gamma_1,\ldots$, $\gamma>\gamma_i$ if $t_i$ is proper and $\gamma\geq\gamma_i$ otherwise, then $\omega^\gamma\cdot a> \bigoplus_{i}\beta_i$, where $\beta_i$ is defined as above.

\begin{example}
	Consider the following illustrative example. To the sequence 
	\[ A(2,3)\to A(1,A(2,2))\to A(1,A(1,A(2,1)))\to\ldots \]
	we assign the ordinals
	\[ \omega^{\omega 2+3}\cdot 2 > \omega^{\omega2+3}+\omega^{\omega2+2}\cdot 2> \omega^{\omega2+3}+\omega^{\omega2+2}+ \omega^{\omega2+1}\cdot 2>\ldots       \]
	This is a descending sequence in $\omega_3$. Here, $a=2$. For a descending sequence in $\omega^\omega$ use base $b=3$ instead of $\omega$. We have $b^{\omega^2}=\omega^\omega$. Replace $\omega^{\omega x+y}a$ with $\omega^x\cdot b^y\cdot a$.
\end{example}

Following  Ben-Amram \cite{Ben-Amram}, there exists a bound $m\in\N$ such that every finite multipath $M=G_0,\ldots,G_n,\ldots$ of length $\geq m$ is \tb{foldable}, where $M$ is foldable if it can be decomposed into three multipaths $M=ABC$ with $H=\overline{B}=\overline{C}=\overline{BC}$, where $\overline{M}$ is the composition of the graphs in $M$.  Note that $H$  is idempotent. In particular, the source and the target functions of $H$ coincide. 
The idea is to assign to each subterm of $t_n$ of the form $g(\mb s)$ with $g\in\fun$  an ordinal of the form $\gamma(\mb u)<\omega^\omega$, where  $\mb u=\mb u_0,\mb u_1,\ldots$ is a finite sequence of tuples appearing in a state transition sequence  $(f_0,\mb u_0)\xrightarrow{G_0}(f_1,\mb u_1)\xrightarrow{G_1}\ldots$ of length $<m$. 
In a one-step reduction we might either extend or contract a finite state transition sequence $(f_0,\mb u_0)\xrightarrow{G_0}(f_1,\mb u_1)\xrightarrow{G_1}\ldots$ into another finite state transition sequence $(g_0,\mb v_0)\xrightarrow{H_0}(g_1,\mb v_1)\xrightarrow{H_1}\ldots$. The second case arises when the corresponding multipath becomes foldable. In the first case, the sequence $\mb v=\mb v_0,\mb v_1,\ldots$ properly extends the sequence $\mb u=\mb u_0,\mb u_1,\ldots$. In the second case, the sequence  $\mb v$ is lexicographically smaller than the sequence $\mb u$. It turns out that we can map sequences $\mb u$ of bounded length to ordinals $\gamma(\mb u)$ in $\omega^\omega$ so that in both cases $\gamma(\mb u)>\gamma(\mb v)$.  For the sake  of exposition we say that the sequence $\mb u$ is \tb{above} the sequence $\mb v$.

The existence of a bound on the length of foldable multipaths is an easy application of finite Ramsey's theorem. We thus have the following:

\begin{lemma}[$\RCA$]\label{bound}
	Let $\mcal G$ be a finite set of size-change graphs. Then there exists $m\in\N$ such that every multipath $M$ in $\cl(\mcal G)$ of length $\geq m$ is foldable. 
\end{lemma}
\begin{proof}
	Finite Ramsey's theorem for pairs.
\end{proof}

Note that the above lemma applies to every  $\mcal G$. The ISCT assumption ensures that the idempotent size-change graph $H$ in the definition of foldable multipath contains a strict arc of the form $\arcd{x}{x}$.

In the next subsection we show how to map  sequences $\mb u$ of bounded length to ordinals $\gamma(\mb u)<\omega^\omega$ so that if $\mb u$ is above $\mb v$ then $\gamma(\mb u)>\gamma(\mb v)$. The reader may skip this part on a first reading.

\subsection{Aboveness}
Let $p \in \N$ be fixed. 

\begin{definition}
	Given a sequence $\mathbf{u}\in\NN^{<p}$, let $\mb u_p\in (\N \cup\{\omega\})^p$ be the sequence of length $p$ which is obtained from $\mb u$ by adding $p-\text{length}(\mathbf{x})$-many $\omega$. That is,
	\[
	\mathbf{u}_p := \mathbf{u} \conc \omega^{(p - \text{length}(\mathbf{u}))}, 
	\]
	where for any natural number $n$, $\omega^{(n)}$ is the sequence of length $n$ with constant value $\omega$. 
\end{definition}

It is easy to see that $\mb u$ is above $\mb v$ if and only if $\mb u_p >_{\text{lex}_{p}} \mb{v}_p$,  	where $\text{lex}_{p}$ is the standard lexicographic order of $(\omega+1)^{p}$.

\begin{definition}
	Given a sequence $\mathbf{u} \in \omega^{<p}$, define
	\[
	\gamma_p(\mb u) := \bigoplus_{i=0}^{p-1}\omega^{p-1-i}(2\cdot\mathbf{u}_p(i)).
	\]
\end{definition}

\begin{lemma}[$\RCAo$]
	Let $p \in \N$ and $\mathbf{u}\in\NN^{<p}$. For any $j < p$, 
	\[
	\bigoplus_{i=j}^{p-1}\omega^{p-1-i}(2\cdot\mathbf{u}_p(i)) < \omega^{p-j} \cdot 2. 
	\]
\end{lemma}
\begin{proof}
	We prove it by induction on $p-j+1$. If $j=p-1$, the first sum is empty and the thesis follows. Assume that the claim holds for $j+1$, we prove it  for $j$. Note that $\omega^{p-1-j}(2\cdot\mathbf{u}_p(j)) \leq  \omega^{p-j}$. Moreover by induction hypothesis
	\[
	\bigoplus_{i=j+1}^{p-1}\omega^{p-1-i}(2\cdot\mathbf{u}_p(i)) < \omega^{p-j-1} \cdot 2. 
	\]
	Therefore 
	\[
	\bigoplus_{i=j}^{p-1}\omega^{p-1-i}(2\cdot\mathbf{u}_p(i)) < \omega^{p-j} \oplus  \omega^{p-j-1} \cdot 2 < \omega^{p-j} \cdot 2. 
	\]
\end{proof}

\begin{lemma}[$\RCAo$]
	Let $p \in \omega$ and $\mathbf{u}, \mathbf{v} \in \N^{<p}$. If $\mb u$ is above $\mb v$ then $\gamma_p(\mathbf{u}) > \gamma_p(\mathbf{v})$. 
\end{lemma}
\begin{proof}
	If $\mathbf{u}$ is  above  $\mathbf{v}$ then $\mb u_p  >_{\text{lex}_{p}} \mathbf{v}_p$.
	 Therefore there exists $j \in p$ such that 
	\[
	(\forall i < j)(\mathbf{u}_p(i) = \mathbf{v}_p(i) \wedge \mathbf{u}_p(j) > \mathbf{v}_p(j)).
	\]
	By the lemma above 
	\[
	\bigoplus_{i=0}^{p-1}\omega^{p-1-i}(2\cdot\mathbf{u}_p(i)) \geq \bigoplus_{i=0}^{j-1}\omega^{p-1-i}(2\cdot\mathbf{v}_p(i)) \oplus \omega^{p-j}(2\cdot (1+\mathbf{v}_p(j))) > \bigoplus_{i=0}^{p-1}\omega^{p-1-i}(2 \cdot \mathbf{v}_p(i)).
	\]
\end{proof}

\subsection{From  $\WO(\omega_3)$ to soundness}

We first give the following definition of stem. Note that this is made in $\RCA$. 
\begin{definition}
	
Let $t \rightarrow_P s$ with $s= t[f(\mathbf{u})]_\rho$. For every subterm $g(\mb s)$ of $s$ with $g\in\fun$ there exists a unique position $\sigma$ of $t$, called the \tb{stem} of $g(\mb s)$, such that  $t|_\sigma = h(\mathbf{t})$, $h\in\fun$,  and one of the following holds:
	\begin{itemize}
		\item $\sigma \perp \rho$ and  $g(\mathbf{s})= h(\mathbf{t})$;
		\item $\sigma \subset \rho$ and 
		$g(\mb s) = h(\mb s)$.
		In this case $h(\mathbf{t})$ is not proper and $\mb t \to_P \mb s$; 
		\item $\sigma = \rho$, and so $h(\mb t) = f(\mathbf{u})$, and 
		$g(\mb s) = e^f|_\tau(\mathbf{u})$
		for some $\tau: f \to g$. In this case $(f, \mathbf{u}) \xrightarrow{\tau}(g, \mb s)$.
	\end{itemize}
\end{definition}

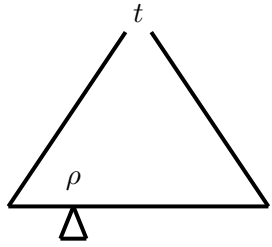
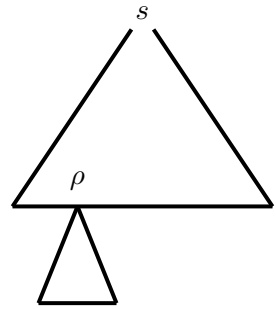
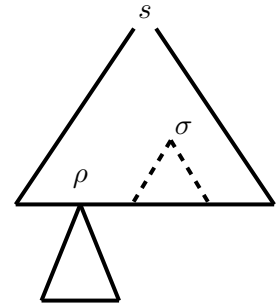
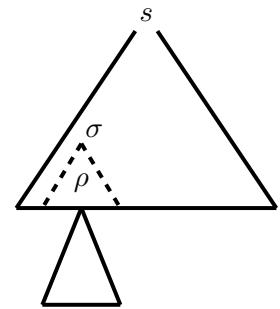
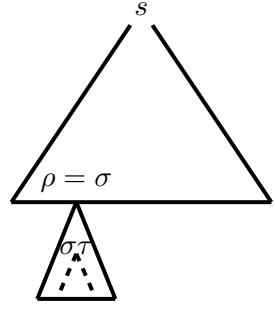
\[
\begin{tikzpicture}[scale=1.71]
\node (1)       at (0,1.5)  {$t$};
\node (2)       at (-0.5,0.2)  {$\rho$};

\draw [line width=1.5, color = black](1)--(-1,0);
\draw [line width=1.5, color = black](1)--(1,0);
\draw [line width=1.5, color = black](-1,0)--(1,0);
\draw [line width=1.5, color = black](-0.5,0)--(-0.6,-0.25);
\draw [line width=1.5, color = black](-0.5,0)--(-0.4,-0.25);
\draw [line width=1.5, color = black](-0.6,-0.25)--(-0.4,-0.25);

\node (3)       at (4,1.5)  {$s$};
\node (4)       at (3.5,0.2)  {$\rho$};

\draw [line width=1.5, color = black](3)--(3,0);
\draw [line width=1.5, color = black](3)--(5,0);
\draw [line width=1.5, color = black](3,0)--(5,0);
\draw [line width=1.5, color = black](3.5,0)--(3.8,-0.75);
\draw [line width=1.5, color = black](3.5,0)--(3.2,-0.75);
\draw [line width=1.5, color = black](3.8,-0.75)--(3.2,-0.75);
\end{tikzpicture}
\]

\[
\begin{tikzpicture}[scale=1.7]
\node (1)       at (0,0)  {{\Large$\sigma \perp \rho $}};
\node (0)       at (-1,0)  { };

\node (3)       at (4,1.5)  {$s$};
\node (4)       at (3.5,0.2)  {$\rho$};
\node (5)       at (4.3,0.6)  {$\sigma$};

\draw [line width=1.5, color = black](3)--(3,0);
\draw [line width=1.5, color = black](3)--(5,0);
\draw [line width=1.5, color = black](3,0)--(5,0);
\draw [line width=1.5, color = black](3.5,0)--(3.8,-0.75);
\draw [line width=1.5, color = black](3.5,0)--(3.2,-0.75);
\draw [line width=1.5, color = black](3.8,-0.75)--(3.2,-0.75);

\draw [line width=1.5, color = black, dashed](4.2,0.5)--(4.5,0);
\draw [line width=1.5, color = black, dashed](4.2,0.5)--(3.9,0);
\end{tikzpicture}
\]

\[
\begin{tikzpicture}[scale=1.71]
\node (1)       at (0,0)  {{\Large$\sigma \subset \rho $}};
\node (0)       at (-1,0)  { };

\node (3)       at (4,1.5)  {$s$};
\node (4)       at (3.5,0.2)  {$\rho$};
\node (5)       at (3.6,0.6)  {$\sigma$};

\draw [line width=1.5, color = black](3)--(3,0);
\draw [line width=1.5, color = black](3)--(5,0);
\draw [line width=1.5, color = black](3,0)--(5,0);
\draw [line width=1.5, color = black](3.5,0)--(3.8,-0.75);
\draw [line width=1.5, color = black](3.5,0)--(3.2,-0.75);
\draw [line width=1.5, color = black](3.8,-0.75)--(3.2,-0.75);
\draw [line width=1.5, color = black, dashed](3.5,0.5)--(3.8,0);
\draw [line width=1.5, color = black, dashed](3.5,0.5)--(3.2,0);
\end{tikzpicture}
\]

\[
\begin{tikzpicture}[scale=1.71]
\node (1)       at (0,0)  {{\Large$\sigma = \rho $}};
\node (0)       at (-1,0)  { };

\node (3)       at (4,1.5)  {$s$};
\node (4)       at (3.5,0.15)  {$\rho=\sigma$};
\node (5)       at (3.5,-0.35)  {$\sigma\tau$};

\draw [line width=1.5, color = black](3)--(3,0);
\draw [line width=1.5, color = black](3)--(5,0);
\draw [line width=1.5, color = black](3,0)--(5,0);
\draw [line width=1.5, color = black](3.5,0)--(3.8,-0.75);
\draw [line width=1.5, color = black](3.5,0)--(3.2,-0.75);
\draw [line width=1.5, color = black](3.8,-0.75)--(3.2,-0.75);
\draw [line width=1.5, color = black, dashed](3.5,-0.4)--(3.65,-0.75);
\draw [line width=1.5, color = black, dashed](3.5,-0.4)--(3.35,-0.75);
\end{tikzpicture}
\]

\begin{proof}[Proof of Theorem \ref{upper bound}.]
	Let $\mathcal G$ be a safe ISCT description of a program $P$, i.e., $\mcal G$ is safe for $P$ and every idempotent graph $G\in\cl(\mcal G)$ contains a strict arc $\arcd{x}{x}$. 
	
	Suppose we are given an infinite reduction sequence \[ f_0(\mb u_0)=t_0\to_P t_1\to_P \ldots \to_P t_n\to_P\ldots \] 
	
	We will assign  ordinals \[\omega_3>\alpha_0\geq \alpha_1\geq\ldots\geq \alpha_n>\ldots \]  
	and prove that $\alpha_n>\alpha_{n+1}$ for infinitely many $n$.  Indeed, we will have 
	$\alpha_n>\alpha_{n+1}$ for every proper reduction $t_n\to_p t_{n+1}$. Note that in any infinite reduction sequence there must be infinitely many proper reductions (exercise). 
	
	If $m$ bounds the length of any non-foldable multipath in $\cl(\mcal G)$ and $r$ is the maximum arity of $f$ for $f\in\fun$, set $p=(m+1)\cdot r$.  Let $a$ be the the maximum size of $e^f$ for $f\in\fun$. From now on we identify a sequence $\mb u$ of length $<p$ with the ordinal  $\gamma_p(\mb u)\in\omega^\omega$.

	By primitive recursion we want to define for all $n$ and for every subterm $h(\mb{t})$ of $t_n$ with $h\in\fun$ a finite multipath  $M=G_0,G_1\ldots,G_{l-1}$ in $\cl(\mcal G)$ of length $<m$ and a sequence
	$\mb u=\mb u_0,\mb u_1,\ldots,\mb u_l$ of length $<p$ such that:
	\begin{enumerate}[(a)]
		\item $(f_0,\mb u_0)\xrightarrow{G_0}(f_1,\mb u_1)\xrightarrow{G_1}\ldots \xrightarrow{G_{l-1}}(f_l,\mb u_l)$, where $G_i:f_i\to f_{i+1}$.
		\item If  $h(\mb t)$ is proper, then  $(f_l,\mb u_l)=(h,\mb t)$.
		\item  If $h(\mb t)$ is not proper,  we also specify $\tau:f_l\to h$ such that  $(f_l,\mb{u}_l)\xrightarrow{\tau} (h,\mb{t})$.
	\end{enumerate}
	
	Note that $(f_l,\mb{u}_l)\xrightarrow{\tau} (h,\mb{t})$ is a $\Sigma^0_1$-condition. We assign to every such subterm $h(\mb t)$ the ordinal $\omega^{\mb u}$ if $h(\mb t)$ is not proper, and the ordinal  $\omega^{\mb u}\cdot a$ otherwise. Finally, we let $\alpha_n$ be the natural sum of all these ordinals. \smallskip
	
	\tb{Construction}.
	
	Stage $n=0$. We have only the term $f_0(\mb u_0)$. Let $M$ be the multipath of length $0$ consisting of $\var(f_0)$ and $\mb u=\mb u_0$. Conditions (a)--(c) are trivially satisfied. \smallskip
	
	Stage n+1. Let $t_{n+1}= t_n[f(\mb{u})]_\rho$.  We can assume by $\Sigma^0_1$-induction that for every non-proper subterm $h(\mb t)$ of $t_n$ the corresponding $\tau:f_l\to h$ is as above.  Let $g(\mb{s})$ be a subterm of $t_{n+1}$. We want to assign a pair $N, \mb v$. Let $\sigma$ be the stem of $g(\mb s)$ and $M,\mb u$ be the pair associated with $h(\mb t)=t_n|_\sigma$. 
	
	\begin{remark}
		If $\sigma\subseteq\rho$, we can specify $\tau:f_l\to g$ such that  $(f_l,\mb u_l)\xrightarrow{\tau}(g,\mb s)$. In fact,
		\begin{itemize}
			\item if $\sigma\subset\rho$, i.e., $h(\mb t)$ is not proper, then $(f_l,\mb u_l)\xrightarrow{\tau}(h,\mb t)$, where $\tau$ has been specified earlier in the construction by (c), and therefore $(f_l,\mb u_l)\xrightarrow{\tau}(g,\mb s)$. Note in fact that $h=g$ and $\mb t\to_P \mb s$;
			\item if $\sigma=\rho$, i.e., $h(\mb t)$ is proper and equal to $f(\mb u)$, then $(f_l,\mb u_l)=(f,\mb u)$ by (b), and therefore $(f_l,\mb u_l)\xrightarrow{\tau}(g,\mb s)$, where $\tau$ is such that $g(\mb s)=e^f|_\tau(\mb u)$. 
		\end{itemize}
	\end{remark}
	
	Case 1. $\sigma\perp \rho$ or $g(\mb s)$ is not proper.  Do nothing, i.e.,  $N=M$ and $\mb v=\mb u$.\smallskip
	
	Case 2. $\sigma\subseteq \rho$ and $g(\mb s)$ is proper, say $g(\mb s)=g(\mb w)$. Consider the multipath $MG$, where $G$ is the size-change graph for the call $\tau:f_l\to g$, where $\tau$ is as in the remark above.  \smallskip
	
	Sub-case 1. $MG$ has length $<m$. Let  $N=MG$, and $\mb v= \mb u,\mb w$. Note that $\mb u$ is above $\mb v$.\smallskip
	
	Sub-case 2. $MG$ is foldable, say $MG=ABC$, with $H=\overline{B}=\overline{C}=\overline{BC}$. Then we fold $M$, i.e., we let $N=AH$. Suppose that $B=G_i,\ldots,G_{j-1}$. Note that $g=f_i=f_j$ and $H:g\to g$.   Now, $H$ is idempotent, and so by ISCT  contains a strict arc $\arcd{x_k}{x_k}$ for some $k$, where $\var(g)=\{x_0,x_1\ldots,\}$.  We define $\mb v$ according to whether  $\arcd{x_k}{x_k}\in H$ or not.  Let 
	\[\mb v=\mb{u}_0,\ldots, \mb{u}_{i-1}, \mb v_i,\mb{w}, \] 
	where  the $k$-element of $\mb v_i$ equals the $k$-element of $\mb u_j$ if $\arcd{x_k}{x_k}\in H$, and the $k$-element of $\mb u_i$ otherwise. Note that we have $\mb u_i\geq \mb v_i$ coordinate-wise. Also, $\mb v_i$ is lexicographically smaller than $\mb u_i$ and therefore $\mb u$ is above $\mb v$. \smallskip
	
	\tb{Verification}. 
	
	Let us first check that in all cases $N,\mb v$ satisfy conditions (a)-(c). 
	
	Case 1. Clearly (a) holds. The only interesting case is when $\sigma\subseteq \rho$. In this case when $g(\mb s)$ is not proper we have (c) by the remark above.
	
	Case 2. Clearly (b) holds. We need only to check  (a).
	
	Sub-case 1. We just need to show that $(f_l,\mb u_l)\xrightarrow{G}(g,\mb w)$. This follows from the remark above and the safety of $\mcal G$. 
	
	Sub-case 2. We just need to show that $(f_{i-1},\mb u_{i-1})\xrightarrow{G_{i-1}}(g,\mb v_i)\xrightarrow{H}(g,\mb w)$. 
	The first state transition follows from the fact that   $(f_{i-1},\mb u_{i-1})\xrightarrow{G_{i-1}}(g,\mb u_i)$ and $\mb u_i\geq \mb v_i$. The second state transition  follows from the safety of $\mcal G$ and the fact that  both $(g,\mb u_i)\xrightarrow{\overline{BC}}(g,\mb w)$ and $(g,\mb u_j)\xrightarrow{\overline{C}}(g,\mb w)$ hold, and $H=\overline{BC}=\overline{C}$. \smallskip

	\begin{claim}
		For all $n$, $\alpha_n\geq\alpha_{n+1}$, and $\alpha_n>\alpha_{n+1}$ for infinitely many $n$.
	\end{claim}
	
	For every position $\sigma$ of a subterm $h(\mb t)$ of $t_n$ with $h\in\fun$, let $\alpha_\sigma$ be the ordinal corresponding to $h(\mb t)$, $S_\sigma$ be the set of subterms $g(\mb s)$ of $t_{n+1}$ with stem $\sigma$, and $\beta_\sigma$ be the natural sum of  ordinals corresponding to  terms in $S_\sigma$. Note the we may have $S_\sigma=\emptyset$. In such a case let $\beta_\sigma=0$. On the other hand, every subterm $g(\mb s)$ of $t_{n+1}$ with $g\in\fun$ belongs to some $S_\sigma$.   Thus $\alpha_n=\oplus_{\sigma}\alpha_\sigma$ and $\alpha_{n+1}=\oplus_{\sigma}\beta_\sigma$.

	We claim that  $\alpha_\sigma\geq \beta_\sigma$  for every position $\sigma$ and hence $\alpha_n\geq\alpha_{n+1}$. 
	
	Case 1. If $h(\mb t)$ is not proper, then  $\alpha_\sigma=\omega^{\mb u}$ and $|S_\sigma|=1$. Say $S_\sigma=\{g(\mb s)\}$.  We assign ordinal $\omega^{\mb u}$ if $g(\mb s)$ is not proper, and $\omega^{\mb v}\cdot a$ otherwise. In the latter case $\mb u$ is almost above $\mb v$, and so $\alpha_\sigma>\beta_\sigma$.
	
	Case 2. If $h(\mb t)$ is proper and $\sigma\neq\rho$, then $\alpha_\sigma=\omega^{\mb u}\cdot a$, $S_\sigma=\{h(\mb t)\}$ and $\alpha_\sigma=\beta_\sigma$. 
	
	Case 3. If $\sigma=\rho$, that is $h(\mb t)$ is the  $f(\mb u)$ in the function reduction from $t_n$ to $t_{n+1}$, then  $\alpha_\rho=\omega^{\mb u}\cdot a$ and $|S_\rho|\leq a$. Note that if $e^f(\mb u)\in \N$ then $S_\rho=\emptyset$, and so $\alpha_\rho>\beta_\rho$.  Otherwise, each $g(\mb s)$ in $S_\rho$ is either not proper, in which case we assign the ordinal $\omega^{\mb u}$, or proper, in which case we assign an ordinal $\omega^{\mb v}\cdot a$, where $\mb u$ is almost above $\mb v$, and so $\omega^{\mb u}>\omega^{\mb v}\cdot a$. In both cases we have $\alpha_\rho>\beta_\rho$.

	Note that Case 1 might occur in any reduction and hence we can have  $\alpha_n>\alpha_{n+1}$ even if the reduction is not proper. However, Case 3 occurs in every proper reduction  and  $\alpha_n>\alpha_{n+1}$ for every such reduction. The claim follows.
\end{proof}

In particular Theorem \ref{upper bound} shows that any descending sequence of ordinals associated to some computation of the generalized P\'{e}ter-Ackermann function $A^n_f$ (as defined in Section \ref{genP-A}) is bounded by some ordinal of the form $\omega^{\omega^{b_n}}$ for some natural number $b_n$. Anyway such $b_n$ depends on the bound on the length of foldable multipaths provided in Lemma \ref{bound} by an application of the finite Ramsey's theorem for pairs. Since uniform bounds for the finite Ramsey's theorem for pairs are rather large, so are the bounds $b_n$ extracted from our proof. These are definitely larger than $\omega^{\omega^n}$, the ordinal which corresponds to $A^n_f$ (see Proposition \ref{wo tot} and Proposition \ref{FA}).

\subsection{Upper bound for tail-recursive programs}\label{tail-recursive}

In this subsection we consider tail-recursive programs. A program function definition is tail-recursive is the recursive call occurs only once and it is the most external function. For instance:
\[
h(\mathbf{t}) = h(f_0(\mathbf{t}), \dots f_{n-1}(\mathbf{t})). 
\]

A program is tail-recursive if every function definition is tail-recursive and there is no mutual recursion. There exists a direct transition-based translation into transition-based programs (see, e.g., \cite{Jones}) and tail-recursive programs are often easy to handle in implementations. 

The goal of this section is to show that  the functional programs which are tail-recursive and ISCT compute exactly the primitive recursive functions. On the one hand, all primitive recursive functions can be computed by simple tail-recursive programs which are ISCT (e.g. see \cite{SCTstar}). On the other hand, Ben-Amram in \cite{BenAmram} has already proved that the first order functional programs defined without nested recursion which are ISCT compute primitive recursive functions. Since tail-recursive programs do not allow nested recursion our result is a corollary of \cite{BenAmram}.  In \cite{SCTstar} there is a different proof which uses an analysis of the intuitionistic proof of the Termination Theorem. As a side result of this analysis, some large bounds are extracted. By following a completely different approach, which follows closely the proof of Theorem \ref{upper bound}, we provide a new proof and we extract the corresponding bounds.

\begin{proposition}\label{Proposition:tail-recursive}
	$\WO(\omega^\omega)$ implies the termination of every tail-recursive ISCT program.
\end{proposition}
\begin{proof}
	We follow the argument of Theorem \ref{upper bound}, but we assign a different ordinal to every subterm $g(s)$ of $t_n$.   Given a tail-recursive functional program $P$ let $\bp{g_0, \dots,g_{k-1}}$ be a fixed ordering between the functions of $P$ such that for every $i<j<k$ $g_j$ does not occur in the expression defining the function $g_{i}$ (note that such an ordering exists since there is no mutual recursion in $P$).
	
	Given a term $g_j(\mathbf{t})$ and $\mathbf{u}$, where $j$ is the index with respect to our fix ordering, we assign to $g_j(\mathbf{t})$ the ordinal $\omega^{pj}(\omega \cdot \gamma(\mathbf{u}) +1 )$ if $g_j(\mathbf{t})$ is proper and the ordinal $\omega^{pj}(\omega \cdot \gamma(\mathbf{u}))$ if $g_j(s)$ is not proper.
	
	Now following the schema of the verification for the bound as presented in the proof of Theorem 6.1  we have the following cases:
	\begin{itemize}
		\item If $g_j(\mathbf{t})$ is not proper, then $\alpha_\sigma = \omega^{pj}(\omega \cdot \gamma(\mathbf{u}))$ and $S_\sigma = \bp{g_i(\mathbf{s})}$ for some $i\leq j$. In both the possible cases for $g_i(\mathbf{s})$ we have  $\alpha_\sigma > \beta_\sigma $, since $\omega \cdot \gamma(\mathbf{u})$ is a limit ordinal. 
		\item If $g_j(\mathbf{t})$ is proper and $\sigma \neq \rho$ then $S_\sigma=\bp{g_j(\mathbf{t})}$. Therefore $\alpha_\sigma =\beta_\sigma$. 
		\item If $\sigma =\rho$. If $g_j(\mathbf{t})$ is proper, since the program is tail recursive we have in the worst case that $S_\sigma = \bp{f_0(t), ... f_{n-1}(t), g_{j'}(f_0(t),..., f_{n-1}(t))}$ for some functions $f_0, \dots, f_n$ whose level which respect to our ordering is less than $j$ and $j' \leq j$. 
		Since $f_0, \dots  f_{n-1}$ have index less than $j$ with respect to our fixed ordering, we associate either $\omega^{p{h_i}}(\omega \cdot \gamma(\mathbf{u}) )$ (if not proper) or $\omega^{p{h_i}}(\omega \cdot \gamma(\mathbf{v_i}) +1 )$ (if proper) to them for some $h_i< j$ and some $\mathbf{v}_i$ such that $\mathbf{u}$ is almost above $\mathbf{v}_i$. If $g_{j'}(f_0(\mathbf{t}),..., f_{n-1}(\mathbf{t}))$ is not proper we associate $\omega^{p{j'}}(\omega \cdot \gamma(\mathbf{u}))$ to it, otherwise we associate $\omega^{p{j'}}(\omega \cdot \gamma(\mathbf{v})+1)$ for some $\mathbf{v}$ such that $\mathbf{u}$ almost above $\mathbf{v}$. Since $j \geq j'$ we have  $\alpha_\sigma > \beta_\sigma $.  \qedhere
	\end{itemize}
\end{proof}

Now assume that $P$ is a tail-recursive program as above and let $g_i$ be the entry function of $P$. By the proposition above every computation from $g(x_0,\dots, x_{n-1})$ corresponds to a descending sequence of ordinals below $\omega^{pi+1}$. We claim that any computation of $g_i(x_0, \dots, x_{n-1})$ has length less than $F_{{pi+2},f}(0)$, for $f(x)= 2x + 2 + \max\{x_0, \dots, x_{n-1},p\}$. 

To prove this we directly adapt the proof of Proposition \ref{wo tot}. Note that for every $\alpha \in \omega^{pi+1}$ we have $\mathrm{mc}(\alpha) \leq \max\{a_i : i \in n\} \cup\{p\}$. Moreover, if $\mathbf{v}$ is obtained after a step from $\mathbf{u}$ in the proof of Proposition \ref{Proposition:tail-recursive}, then $\max\{v_j : j < p+1\} < \max\{u_j: j< p+1\}+1$, therefore $f(x) > \mathrm{mc}(\alpha_{x+1}) + x+1$, as required in the proof of Proposition \ref{wo tot}.  Observe that the proof of Proposition \ref{wo tot} guarantees that for every infinite decreasing sequence  $\alpha_i$ we get that every $\beta_i$ is positive. We can straightforwardly derive from this argument that if the sequence of $\alpha_i$ has length $n$, then we $\beta_i$ is positive for every $i < n$. 

Assume that we have a decreasing sequence of ordinals below $\omega^{pi+1}$ which is derived from a computation of length greater than $F_{{pi+2},f}(0)$. Hence we would get that the corresponding $\beta_i$ are positive for every $i \leq F_{{pi+2},f}(0)$. But this provides a contradiction, by definition of $\beta_i$.

Observe that, in general, these bounds seem to be huge. For instance let us consider the toy-program analyzed in \cite{SCTstar}:
	\[
	\begin{split}
	f(x,y,\temp,\esp,z) = &\mbb{ if } (y=0) \mbb{ then } 1\\
	&\mbb{ else if } (y=1) \mbb{ then } \esp\\
	&\qquad \mbb{ else } \tau_1: f(x,y-1,*, \tau_2: g(x,y,0,\esp,x),*)\\
	g(x,y,\temp,\esp,z) = &\mbb{ if } (z=0) \mbb{ then } 0\\
	&\mbb{ else if } (z=1) \mbb{ then } \temp\\
	& \qquad\mbb{ else }  \tau_0: g(*,*,\temp+\esp,\esp,z-1)\\	
	\end{split}			 
	\]
	where $*$ denotes any value. Note that $f(x,y,0,1,z)$ computes $x^y$. Every size-change graph corresponds to some composition of $G_{\tau_0}:g \to g$, $G_{\tau_1}:f\to f$ and $G_{\tau_2}:f \to g$. 
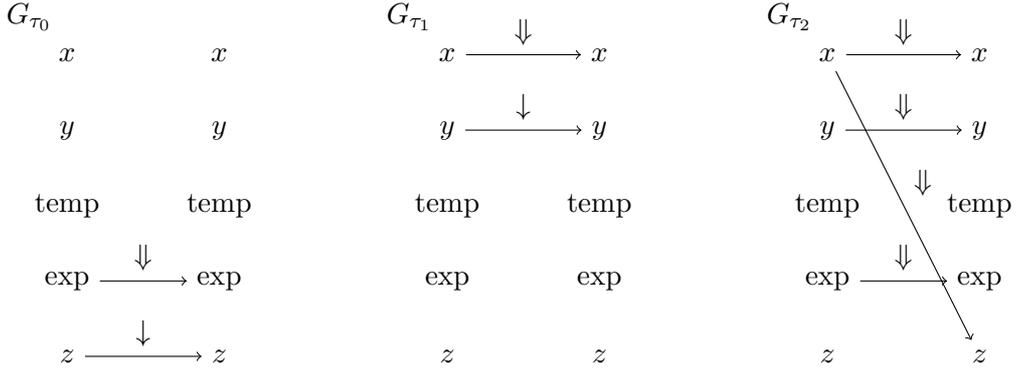
	\[
	\begin{tikzpicture}[xscale=1,yscale=1]
	\node (G0) at (-0.5, 4.5) {$G_{\tau_0}$};
	\node (x1) at (0, 4) {$x$};
	\node (y1) at (0, 3) {$y$};
	\node (temp1) at (0, 2) {$\temp$};
	\node (exp1) at (0, 1) {$\esp$};
	\node (z1) at (0, 0) {$z$};
	
	\node (x2) at (2, 4) {$x$};
	\node (y2) at (2, 3) {$y$};
	\node (temp2) at (2, 2) {$\temp$};
	\node (exp2) at (2, 1) {$\esp$};
	\node (z2) at (2, 0) {$z$};
	
	\node (G1) at (4.5, 4.5) {$G_{\tau_1}$};
	\node (x3) at (5, 4) {$x$};
	\node (y3) at (5, 3) {$y$};
	\node (temp3) at (5, 2) {$\temp$};
	\node (exp3) at (5, 1) {$\esp$};
	\node (z3) at (5, 0) {$z$};
	
	\node (x4) at (7, 4) {$x$};
	\node (y4) at (7, 3) {$y$};
	\node (temp4) at (7, 2) {$\temp$};
	\node (exp4) at (7, 1) {$\esp$};
	\node (z4) at (7, 0) {$z$};
	
	\node (G2) at (9.5, 4.5) {$G_{\tau_2}$};
	\node (x5) at (10, 4) {$x$};
	\node (y5) at (10, 3) {$y$};
	\node (temp5) at (10, 2) {$\temp$};
	\node (exp5) at (10, 1) {$\esp$};
	\node (z5) at (10, 0) {$z$};
	
	\node (x6) at (12, 4) {$x$};
	\node (y6) at (12, 3) {$y$};
	\node (temp6) at (12, 2) {$\temp$};
	\node (exp6) at (12, 1) {$\esp$};
	\node (z6) at (12, 0) {$z$};

	\path (exp1) edge [->]node [auto] {${\Downarrow}$} (exp2);
	\path (z1) edge [->]node [auto] {${\downarrow} $} (z2);
	
	\path (x3) edge [->]node [auto] {${\Downarrow}$} (x4);
	\path (y3) edge [->]node [auto] {${\downarrow} $} (y4);
	
	\path (x5) edge [->]node [auto] {${\Downarrow}$} (x6);
	\path (y5) edge [->]node [auto] {${\Downarrow} $} (y6);
	\path (x5) edge [->]node [auto] {${\Downarrow} $} (z6);
	\path (exp5) edge [->]node [auto] {${\Downarrow} $} (exp6);
	\end{tikzpicture}
	\]
	The idempotent graphs in $\cl(\mathcal{G})$ are $G_{\tau_0}:g\to g$ and $G_{\tau_1}:f \to f$ (since the source and the target of the other size-change graphs are different).
	Hence this program is ISCT. Recall that $p$ is defined to be $(m+1)r$. Note that the maximal arity $r$ for this program is $5$, so $p>5$. As we already mentioned, the bound $m$ for the length of the unfoldable multipaths is provided by an application of the finite Ramsey's theorem and it is well-known that the bounds for the finite Ramsey's numbers are pretty large for $n>4$. Therefore the bounds extracted from this proof involve functions $F_{l,f}$ with $l$ much bigger than $5$, which is extremely loose, since $F_{3, i \mapsto i+1} (\max \{ x, y \})$ is already a bound for the length of the computations of this program. This can be shown as a direct application of the bound for the Termination Theorem provided in  \cite{Schmitz}. 
	

\section{Conclusion}

In this paper we proved that, over $\RCA$,  
\begin{itemize}
	\item ISCT soundness $=$ $\WO(\omega_3)$ (by Corollary \ref{Corollary: lower bound} and by Theorem \ref{upper bound});
	\item  MSCT soundness $ \geq \WO(\omega_2)$ (see Remark \ref{Remark: fromMSCTtoAck}).
\end{itemize}
It is known that $\is2$ implies $\WO(\omega_2)$, but does not imply $\WO(\omega_3)$ (see \cite[Remark 2.4]{SimpsonWO}). Moreover, $\WO(\omega_3)$ does not imply $\is2$ (see \cite[Corollary 4.3]{SimpsonWO}). Since $\is2 + $ MSCT soundness implies $\WO(\omega_3)$, we can conclude that MSCT soundness $> \WO(\omega_2)$. 

Finally, since every MSCT program is also ISCT, provably in $\RCA$, we have that $\WO(\omega_3) \geq$ MSCT soundness.  As discussed in the introduction, we leave  open whether the inequality is strict (see Question \ref{question}). Let $A^n$ be the generalized P\'{e}ter-Ackermann function for the successor function (i.e. $f(x)=x+1$). If we restrict MSCT soundness to the statement we call MSCT$^*$ soundness:  
\[
\forall n (A^n \mbox{ MSCT } \implies A^n \mbox{ terminates}),
\]
then we obtain something strictly weaker than $\WO(\omega_3)$, as one can obtain a model which seperates them, in the following manner:

Starting from a countable, nonstandard model of PA, shorten it (i.e., take an initial segment) to a model $M$, such that $M \models\WO(\omega^{\omega^n})$ if and only if $n$ is standard. In $M$, the Strong Pigeonhole Principle SPP$_n$ fails for all nonstandard $n$, but is true for all standard $n$. So there exists a coloring in $n$ many colors of the natural numbers for which there does not exist a set of colors which appear infinitely many times in this coloring.

This is a model such that $M \models \forall n (A^n \mbox{ MSCT } \implies A^n \mbox{ terminates})$ and $M \models \exists n (A^n \mbox{ does not terminate})$.  Indeed, we have  that $A^n$ is MSCT only if $n$ is standard, as $A^n$ MSCT implies SSP$_{k}$, for $n=2^{k}-2$, as shown in the proof of Proposition \ref{Propositon:MSCTandIS2}. 

Therefore, there is a separation between ISCT Soundness and MSCT$^*$ Soundness. This suggests that a possible direction, to address Question \ref{question}, could be to solve the following:
\begin{question}
	Is MSCT Soundness equivalent to MSCT$^*$ Soundness? 
\end{question}
Of course the direction from MSCT Soundness  to MSCT$^*$ Soundness is trivial, as the latter one is a direct corollary of the former one. The vice versa is still open.

\subsection*{Funding}
The work of the first author was supported by the Funda\c{c}\~{a}o para a Ci\^{e}ncia e a Tecnologia [UID/MAT/04561/2013] and Centro de Matem\'{a}tica,  Applica\c{c}\~{o}es Fundamentais e Investiga\c{c}\~{a}o Operacional of Universidade de Lisboa.

The work of the fourth author is partially supported by
JSPS KAKENHI grant number 16K17640 and JSPS Core-to-Core Program
(A.~Advanced Research Networks).


\end{document}